\documentclass[10pt,letterpaper]{amsart}
\usepackage{hyperref}
\usepackage{amsmath,amssymb,url}
\usepackage{graphicx, color}
\usepackage{mathtools}
\usepackage{enumerate}
\usepackage{latexsym}
\usepackage[margin=1in]{geometry}
 \newcommand{\0}{\mathbf{0}}
 \newcommand{\1}{\mathbf{1}} 
 \newcommand\cA{\mathcal A}
 \newcommand\cB{\mathcal B}
 \newcommand\C{\mathbb C}
 \newcommand\cC{\mathcal{C}}
 \newcommand\cD{\mathcal{D}} 
 \newcommand\dist{\textrm{dist}} 
 \newcommand\cE{\mathcal E}
 \newcommand{\e}{\mathbf{e}} 
 \newcommand{\g}{\mathbf{g}}  
 \newcommand\rH{\textrm{H}}
 \newcommand\cI{\mathcal I}
  \newcommand\bi{\mathbf i}
 \newcommand\cJ{\mathcal J} 
 \newcommand\cK{\mathcal K} 
 \newcommand\cL{\mathcal L} 
 \newcommand\cM{\mathcal M}  
 \newcommand\N{\mathbb N} 
 \newcommand\Q{\mathbb Q}  
  
 \newcommand{\rank}{\textrm{rank }}
 \newcommand{\range}{\textrm{range }}
 \newcommand\cR{\mathcal{R}}
 \newcommand\R{\mathbb R}
  
 \newcommand\rS{\mathrm S}
 \newcommand\supp{\textrm{supp }}
 \newcommand{\trans}{^\top}
 \newcommand{\uu}{\mathbf{u}} 
 \newcommand{\U}{\mathbf{U}}  
 \newcommand{\V}{\mathbf{V}} 
 \newcommand{\vv}{\mathbf{v}} 
 \newcommand{\w}{\mathbf{w}} 
 \newcommand{\bW}{\mathbf{W}} 
 \newcommand{\x}{\mathbf{x}}
 \newcommand{\y}{\mathbf{y}}
 \newcommand{\z}{\mathbf{z}} 

\newtheorem{theorem}{Theorem}[section]
\newtheorem{proposition}[theorem]{Proposition}
\newtheorem{lemma}[theorem]{Lemma}
\newtheorem{corollary}[theorem]{Corollary}
\theoremstyle{definition}
\newtheorem{definition}[theorem]{Definition}
\newtheorem{example}[theorem]{Example}

\theoremstyle{remark}

\newtheorem{Algorithm}[theorem]{Algorithm}

\DeclareMathOperator{\tr}{tr}
\DeclareMathOperator{\diag}{diag}

\begin{document}
\title{{The Collatz-Wielandt quotient  for pairs of nonnegative operators}}
\author{Shmuel~Friedland, Chicago}
\address{Department of Mathematics, Statistics and Computer Science,  University of Illinois, Chicago, IL,  60607-7045,
email: {friedlan@uic.edu}}

\begin{abstract}
In this paper we consider two versions of the Collatz-Wielandt quotient for a pair of nonnegative operators $A,B$ that map a given pointed generating cone in the first space into  a given pointed generating cone in the second space.  If the two spaces and two cones are identical, and $B$ is the identity operator then one version of this  quotient is the spectral radius of $A$.  In some applications, as commodity pricing, power control in wireless networks and quantum information theory, one needs to deal with the Collatz-Wielandt quotient for two nonnegative operators.  In this paper we treat the two important cases: a pair of rectangular nonnegative matrices and a pair completely positive operators.  We give a characterization of minimal optimal solutions and  polynomially computable bounds on the Collatz-Wielandt  quotient.
\end{abstract}

 \maketitle

\keywords\textbf{Keywords}{ Perron-Frobenius theory, Collatz-Wielandt quotient, completely positive operators, commodity pricing, wireless networks, quantum information theory.}

\subjclass{\bf 2010 Mathematics Subject Classification.}
15A22, 15A45, 15B48, 15B57, 94A40

\section{Introduction}\label{sec:intro}
The celebrated Perron-Frobenius theorem describes important spectral properties of a square matrix $A$ with nonnegative entries \cite{Per07, Fro08, Fro09, Fro12}.   In particular, the spectral radius $\rho(A)$, (the maximum of absolute values of all eigenvalues of $A$), is an eigenvalue of $A$.  Furthermore, to $\rho(A)$ correspond a nonnegative eigenvector $\y$:
 \begin{equation}\label{PFeigvec}
 A\y=\rho(A)\y, \quad \y\gneq \mathbf{0},
 \end{equation} 
 which is called the Perron-Frobenius eigenvector, abbreviated as PF-eigenvector. 
 If $A$ is irreducible then $\mathbf{y}>\mathbf{0}$ and is unique up to multiplication by a positive scalar. 
 There are many classical and recent books giving a full account of the Perron-Frobenius theory of nonnegative matrices  for example \cite{BeP79, Fri15,Gan59,HorJoh,Mey00, Min88, Sen}.  It is well known that PF-theory found innumerous applications in all sciences.  See for example \cite{ABHKLPP12,ABHKLPP,PSC05,Sri03} and references therein.
 
 We denote by $\R^{m\times n}\supset \R^{m\times n}_+$ the sets of real valued and nonnegative valued $m\times n$ matrices respectively. ($\R^{m}=\R^{m\times 1}\supset \R_+^m=\R_+^{m\times1}$ the set of column vectors and the subset of nonnegative column vectors with $m$ coordinates respectively.)  Let $[n]=\{1,\ldots,n\}, \frac{0}{0}=0, \frac{+}{0}=\infty$. 
 
  One of the most applicable feature of PF-theory is the Collatz-Wielandt  characterization of $\rho(A)$ for $A\in\R^{m\times m}_+$  \cite{Col42, Wie50}:
 \begin{equation}\label{ColWiel}
 \inf_{\x=(x_1,\ldots,x_m)\trans> \0} \max_{i\in[m]}\frac{(A\x)_i}{x_i}=\rho(A).
 \end{equation}
 Theorem 6.4.5 in \cite{Fri15} gives a necessary and sufficient condiitons on $A$ that the above infimum is achieved for some positive $\x$.   In a simple noiseless model  in wireless netrworks $\frac{1}{\rho(A)}$ is the reception threshold \cite{Sri03, PSC05} and \cite[\S6.9]{Fri15}.
   
 Given a pair of nonsquare matrices $A,B\in \R^{m\times n}$ one can consider the generalized eigenvalue problem
\begin{equation}\label{geneigprobAB}
A\x=\lambda B\x, \quad A,B\in\R^{m\times n}, \lambda\in\C.
\end{equation}
In order to assure that one has a finite number of eigenvalues, one needs to assume that 
 
 \noindent 
$\max(\rank A,\rank B)=n$, which implies that $m\ge n$.
  There is an extensive literature on this problem, see for example \cite{Erd67,BEGM,CG06} and references there in. 
 A first attempt to generalize Perron-Frobenius theory to \eqref{geneigprobAB}, to the best knowledge of the author, is by Mangasarian \cite{Man71}.
 He showed the assumption that $B\trans\y\ge \0$ implies  $A\trans \y\ge \0$ yields that \eqref{geneigprobAB} has a discrete and finite spectrum, and the eigenvalue with the largest absolute value is real, nonnegative and a corresponding eigenvector is nonnegative.
 
 The Perron-Frobenius theory was generalized to nonnegative operators $A$ with respect to a closed pointed generating cone $\mathbf{K}$ in finite and infinite dimensional Banach spaces \cite{KR48,Kar59, Sch74,BeP79}.  There is also a natural generalization of the Collatz-Wielandt characterizations to  the spectral radius of $\rho(A)$ \cite{Fri90,Fri91}.
 
 The aim of this paper is to consider the Collatz-Wielandt type infmax problem for a pair of nonnegative operators $A,B: \R^{N_1}\to \R^{N_2}$,  with respect to closed pointed generating cones $\mathbf{K}_i\subset \R^{N_i}$ for $i=1,2$: $A\mathbf{K}_1,B\mathbf{K}_1\subseteq  \mathbf{K}_2$.   
 Denote by $\mathbf{K}_i^{o}$ the interior of $\mathbf{K}_i$.
 Let 
 \begin{equation}\label{defrABxK}
 r(A,B,\x)=\inf\{t, t\in[0,\infty], tB\x-A\x\in\mathbf{K}_2\} \quad \textrm{for } \x\in\mathbf{K}_1\setminus\{\0\}.
 \end{equation}
 We set $r(A,B,\x)=\infty$ if $tB\x-A\x\not\in\mathbf{K}_2$ for for each $t\ge 0$.
 Define
 \begin{equation}\label{defrhoABK}
 \rho(A,B)=\inf \{r(A,B,\x), \x\in\mathbf{K}_1^{o}\}.
 \end{equation} 
 In general, $\rho(A,B)$ can have any value in $[0,\infty]$.  We call $\rho(A,B)$ the \emph{Collatz-Wielandt quotient}.
 
 We also consider the following variation of $\rho(A,B)$:
 \begin{equation}\label{defhrhoABK}
 \hat\rho(A,B)=\inf \{r(A,B,\x), \x\in\mathbf{K}_1\setminus\{\0\}\}.
 \end{equation} 
 We call $\hat\rho(A,B)$ the \emph{weak Collatz-Wielandt quotient}.
 Clearly, $\hat\rho(A,B)\le \rho(A,B)$.  Even in the classical case, where $A\in\R_+^{m\times m}$ and $B$ is the identity matrix one may have the strict inequality $\hat\rho(A,I)<\rho(A,I)=\rho(A)$.  A simple example is the following one.
 Assume that $A$ is a direct sum of $k$ irreducible matrices $A_1,\ldots,A_k$, where $\rho(A_1)>\rho(A_2)>\cdots > \rho(A_k)$.  (So $A$ is block diagonal diag$(A_1,\ldots,A_k)$.)  Then $\rho(A,I)=\rho(A_1)$ and $\hat\rho(A,I)=\rho(A_k)$.  See \S\ref{sec:clas}.  We show that  if either $A$ or $B$ are positive then the Collatz-Wielandt quotient and the weak Collatz-Wielandt quotient are equal.  Furthermore we have the following stability results. Suppose that  we have two sequence of positive matrices $A_l$ and $B_l$ that converge to $A$ and $B$ respectively.   Then $\lim_{l\to\infty}\rho(A,B_l)=\hat \rho(A,B)$, and $\lim_{l\to\infty} \rho(A_l,B)=\rho(A,B)$ provided that $B$ does not have a zero row. Thus  the Collatz-Wielandt quotient and the weak Collatz-Wielandt quotient seem to be equally important quantities.  
 
 In the first part of this paper we consider the Collatz-Wielandt quotient for a pair of rectangular nonnegative matrices $A,B\in\R^{m\times n}_+$, i.e. :
 \begin{equation}\label{minmaxprb}
\rho(A,B):= \inf_{\x=(x_1,\ldots,x_n)\trans> \0} \max_{i\in[m]}\frac{(A\x)_i}{(B\x)_i}.
 \end{equation}
 (So $\mathbf{K}_i=\R^{N_i}_+$ for $i=1,2$ and $N_1=n, N_2=m$.)
 
 We now give a simple model of commodity pricing, where the above Collatz-Wielandt ratio arises.   (Another example in wireless networks discussed in \cite{ABHKLPP12,ABHKLPP} is discussed in \S\ref{sec:specB}.)  Assume that we have $m$ producers of commodities which produce $n$ commodities.  Each producer $i$ produces a subset of commodities $C(i)\subset [n]$.  (We do not exclude the possibility that two producers produce the same commodity $j$.)  Assume that the price of commodity $j$ is $x_j> 0$.  Then $\x=(x_1,\ldots,x_n)\trans> \0$ is the pricing vector.  The expected value of the cost of the the producer $i$ for one unit of his products is $\sum_{j=1}^n a_{ij}x_j$.  The expected value of the profit of the producer $i$ for one unit  is $\sum_{j=1}^n b_{ij}x_j$.  One can impose the obvious conditions that $a_{ij}=0$ if $j\in C(i)$, (the producer $i$ does not buy the commodity it produces), and $b_{ij}=0$ if $j\not\in C(i)$, (the producer sell only the items it produces),  Then the ratio of the profit to the expense for the producer $i$  is $\frac{(B\x)_i}{(A\x)_i}$.  We call this ratio \emph{profit factor}.
 In order that each producer will stay in business for the pricing vector $\x$ one needs to satisfy the minimum profit  factor requirement: $\min_{i\in[m]}\frac{(B\x)_i}{(A\x)_i}\ge \beta$.
 Then the optimal pricing choice is the solution to the supremum problem  
 \[\sup_{\x>\0}\min_{i\in[m]}\frac{(B\x)_i}{(A\x)_i}=\frac{1}{\rho(A,B)}.\] 
 
 We are interested in a nontrivial case, where $\rho(A,B)<\infty$.  It is easy to show that this inequality holds if and only the following condition is satisfied: For each zero row $i$ of $B$  the row $i$ of $A$ is zero.   
 
  We now summarize our results for the extremal problem \eqref{minmaxprb}.  Assume that $\rho(A,B)\in(0,\infty)$.  (It is easy to characterize the case $\rho(A,B)=0$.)
  Then there exists $\y\in\R^n_+\setminus\{\0\}$ such that $\rho(A,B)=r(A,B,\y)$ with the following property:  There exists a sequence $\y_k>\0$ for $k\in\N$ such that $\lim_{k\to\infty}\y_k=\y$ and $\lim_{k\to\infty}r(A,B,\y_k)=\rho(A,B)=r(A,B,\y)$.   Such $\y$ is called an optimal $\y$.   
 An optimal vector $\y$ is called \emph{minimal optimal} if $\y$ is an optimal vector, and there is no optimal vector  $\z$ whose support is strictly contained in the support of $\y$.
 We call $\y$ a generalized Perron-Frobenius vector, abbreviated as GPF-eigenvector, if
 \begin{equation}\label{eigveccond}
 A\y=\rho(A,B)B\y, \quad \y\gneq \0.
  \end{equation}
  Note that if \eqref{geneigprobAB} has an eigenvector $\x> \0$ with a corresponding $\lambda>0$, then $\rho(A,B)\le \lambda$ and it is easy to give examples where
  $\rho(A,B)<\lambda$ and each optimal $\y$ is not a GPF-eigenvector.  (See the example in the end of \S\ref{sec:minopt}.)  We next show, as briefly pointed in \cite{ABHKLPP}, that for any $\varepsilon\in (0,1)$, we can find one of the following:  Either $\rho(A,B)<\varepsilon$ or we can find an approximation of $\rho(A,B,\varepsilon)$, such that $|\rho(A,B,\varepsilon)- \rho(A,B)|\le \varepsilon \rho(A,B)$,  in polynomial time.  This follows from the well known fact that a solvability of linear system of equations is polynomial in the data \cite{GLS88,Lov86}. 
  
   We show that each minimal optimal $\y$ has at most $m$ positive coordinates.  The existence of an optimal $\y$ with at most $m+1$ positive coordinates in a general setting is shown in \cite{ABHKLPP}.  Furthermore, if there exists an optimal vector with  $\ell\ge m$ positive coordinates, then the rank of the matrix $A'-\rho(A,B)B'$ is less than $m$. (Here $A',B'\in\R^{m\times \ell}_+$ are the submatrices induced by $\ell$ positive entries of $\y$.)
 This result implies that for each minimal optimal $\y'$ with $\ell$ positive coordinates, there is exists a minimal optimal $\y$ with the same support as $\y$ such that $(A\y-\rho(A,B)B\y)_i=0$ for at least $\ell$ indices $i\in[m]$.  That is, there exists a minimal optimal solution that is a GPF-eigenvector of the system $\tilde A\y=\rho(A,B)\tilde B\y$, where $\tilde A,\tilde B$ are the submatrices of $A,B$ obtained from $A,B$ by erasing a set of the of rows $\cI$ in $A,B$ respectively.
 For the optimal commodity pricing model that we introduced above the above results have the following  meaning:  First each zero coordinate $j$ of $\y$ implies that the commodity $j$ is not produced.    The producers corresponding to the set $\cI$ have their profit ration above $\frac{1}{\rho(A,B)}$. For all other producers the profit ratio is $\frac{1}{\rho(A,B)}$.  Similar results are shown for $\hat\rho(A,B)$.  
 
 We also give the following generalization of the main result in \cite{ABHKLPP}.  Namely, if $B$ has no zero row and each column has one positive element, then
 there is an optimal solution which is a GPF-eigenvector.  That is, in the wireless model of transmitters-receivers, where each receiver $i$ can obtain a signal from several transmitters, which can only transmit to the receiver $i$, there is a choice to pick exactly one transmitter $j(i)$.  (However, if the system is not irreducible, as defined in \cite{ABHKLPP}, this choice would imply that some other transmitters to receiver $i'$ should be shut off.)
  
 The second part of this paper is a generalization of the above results to pairs of completely positive operators, which are frequently appear in quantum information theory
 as quantum channels.
 Denote by $\rH_{n}\supset \rH_{+,n}\supset \rH_{+,1,n}$ the real space of $n\times n$ Hermitian matrices, the cone of positive semidefinite matrices and the convex set of positive semidefinite matrices of trace one.  Note that $\rH_{+,n}$ is a pointed generating cone in $\rH_n\equiv\R^{n^2}$.
 In quantum information theory (QIT), $\rH_{+,1,n}$ is the set of density matrices, (mixed states).  Recall that $\cC:\rH_{n}\to \rH_{m}$ is called a completely positive operator, abbreviated as CP-operator, if
 \begin{equation}\label{defccpop}
 \cC(X)=\sum_{j=1}^k T_jXT_j^*, \quad T_j\in\C^{m\times n}, j\in[k].
 \end{equation}
 (Here $\C^{m\times n}$ is the space of $m\times n$ complex valued matrices and $T^*=\bar T\trans$ for $T\in\C^{m\times n}$.)
 Then $\cC(\rH_{+,n})\subseteq \rH_{+,m}$, that is, $\cC$ is a nonnegative operator with respect to the pair of cones $\rH_{+,n},\rH_{+,m}$.  In QIT $\cC$ is called quantum channel if 
 \begin{equation}\label{dwfqc}
 \sum_{j=1}^k T_j^*T_j=I_n.
 \end{equation}
 That is, $\cC$ is a quantum channel if and only if $\cC$ is a CP trace preserving operator.
 In particular, $\cC$ maps a density matrix to a density matrix.  Quantum channel is one of the most significant notions in QIT \cite{Sho02,Hol06, Shi06,Has09, MW09,Hol12}.
The second main problem we discuss are $\rho(\cA,\cB)$ and $\hat\rho(\cA,\cB)$ for two CP-operators $\cA,\cB:\rH_{+,n}\to \rH_{+,m}$.  The quantities $\rho(\cA,\cB)$ and $\hat\rho(\cA,\cB)$ could be viewed the quantum analog of the optimal commodity pricing assignment discussed above.
We show that most of our results on $\hat\rho(A,B)$ generalize to $\hat\rho(\cA,\cB)$, and some results on $\rho(A,B)$  generalize to $\rho(\cA,\cB)$.   We show that there exists weak optimal  $Y\in\rH_{+,n}\setminus\{0\}$ such that $r(\cA,\cB,Y)=\hat\rho(\cA,\cB)$.  A weak optimal $Y$ is called minimal if there is no optimal $Z$ such that $\range Z$ is strictly contained in $\range Y$.
We show that a minimal weak optimal $Y$ has rank at most $m$.  Assume that $Y'$ is a minimal weak optimal with rank $\ell$.  Then there exists a minimal weak optimal $Y$ satisfying
 $\range Y'=\range Y$, such that $\rank (\hat\rho(\cA,\cB)\cB-\cA)(Y)\le m-\ell$.  In particular, if $\ell=m$ then $Y$ is a weak GPF-eigenvector.
Assume that $\cB$ is $\delta$-positive for a given rational $\delta>0$.  (This assumption can be verified in polynomial time.)   Then  $\rho(\cA,\cB)=\hat\rho(\cA,\cB)$.  Furthermore $\rho(\cA,\cB)$ has an $\varepsilon\in (0,1)$ approximation in polynomial time in $\langle\cA\rangle +\langle\cB\rangle+\langle\delta\rangle +\langle\varepsilon\rangle$.  (We need the assumption that $\cB$ is $\delta$-positive  because verifying
 the existence of a nonzero positive semidefinite matrix satisfying $(\cA-t\cB)(X)\le 0$ is a feasibility problem in semidefinite programming, which may be not polynomiallay solvable.)
 
 We now survey briefly the content of the paper.  Section \ref{sec:prelim} discusses basic properties of $\rho(A,B)$ and $\hat\rho(A,B)$. Section \ref{sec:clas} discusses that classical case of the pair $A,B\in\R^{m\times n}_+$ where $m=n$ and $B$ is the idenity matrix $I$.  We show that $\rho(A,I)=\rho(A)$.  If $A$ is not irreducible than one may have the strict inequality $\hat\rho(A,I)<\rho(A,I)$.  We characterize completely $\hat\rho(A,I)$.  In Section \ref{sec:polaprox} we give a polynomial time approximation algorithm to $\rho(A,B)$ and $\hat\rho(A,B)$.  In Section \ref{sec:minopt} we give various properties of minimal optimal and minimal weak optimal vectors for the pair $A,B\in\R^{m\times n}_+$.  In Section \ref{sec:specB} we discuss WN-pairs $A,B\in\R^{m\times n}_+$ arising in wireless network.  That is, $B$ has no zero row and one positive element in each column.  Such pairs were introduced and studied in \cite{ABHKLPP12, ABHKLPP}.  We give generalizations of the results in \cite{ABHKLPP12, ABHKLPP}, since we do not restrict ourselves to $S$-irreducible systems. Sections \ref{sec:CPop}, \ref{sec:polaprhocAB} and \ref{minoptCP} are devoted  to study of the Collatz-Wielandt quotients for pairs of completely positive operators.

 \section{Preliminary results}\label{sec:prelim}
For a positive integer $n$ let  $\1_n=(1,\ldots,1)\trans \in\R^n$.  Let  $S\subseteq [m], T\subseteq [n]$.   Denote by $\1_S=(x_1,\ldots,x_m)\trans\in\R_+^m$ the characterstic vector of $S$, i.e., $x_i=1$ if $i\in S$ and $x_i=0$ otherwise.  So $\1_{\emptyset}=\0, \1_{[m]}=\1_m$.  Assume that $F\in\R^{m\times n}$.  Denote by $F(S,T)$ the matrix obtained from $F$ be deleting the rows of $F$ in the set $S$ and the columns in the set $T$.  So $F(S,T)\in\R^{(m-|S|)\times (n-|T|)}$.  If either $S=[m]$ or $T=[n]$ we denote $F(S,T)$ by $\emptyset$. Denote by $F[S,T]$ the matrix $F([m]\setminus S,[n]\setminus T)$.

Assume that $A,B\in\R^{m\times n}_+$.  For $\x=(x_1,\ldots,x_n)\trans\gneq\0$ we define
 \[r(A,B,\x)=\inf\{t, t\ge 0, A\x\le tB\x\}.\]
 Note that $r(A,B,\x)\in[0,\infty]$.  That is , $r(A,B,\x)=\infty$ if and only if there exists $i\in[m]$ such that $(A\x)_i>0$ and $(B\x)_i=0$.  Equivalently
 \[r(A,B,\x)=\max\{\frac{(A\x)_i}{(B\x)_i}, \; i\in[m]\}.\]
 Hence
 \begin{equation}\label{charrhoAB}
 \rho(A,B)=\inf\{r(A,B,\x), \; \x>\0\}, \; 
 \hat\rho(A,B)=\inf\{r(A,B,\x), \; \x\in\R_+^n\setminus\{\0\}\}.
 \end{equation}
 
 The following lemma is deduced straightforward.
 \begin{lemma}\label{rowB=0}  Assume that $A,B\in\R^{m\times n}_+$.  
 \begin{enumerate}
\item $\rho(0,B)=\hat \rho(0,B)=0$.
\item $\rho(A,B)=\infty$ if and only if $B$ has a zero row $i$, while the i-th row of $A$ is not zero.
 \item Let  $S$ be a strict subset of $[m]$.  
 Then for each $\x\gneq \0$ $r(A(S,\emptyset),B(S,\emptyset),\x)\le r(A,B,\x)$. 
 In particular $\rho(A,B)\ge \rho(A(S,\emptyset),B(S,\emptyset))$ and  
$\hat\rho(A,B)\ge \hat\rho(A(S,\emptyset),B(S,\emptyset))$.  Suppose furthermore that for each $i\in S$  the row $i$ of $A$ and $B$ are zero.  Then for each $\x\gneq \0$ $r(A(S,\emptyset),B(S,\emptyset),\x)=r(A,B,\x)$.  In particular, $\rho(A,B)=\rho(A(S,\emptyset),B(S,\emptyset))$ and 
$\hat \rho(A,B)=\hat \rho(A(S,\emptyset),B(S,\emptyset))$ .
 \item Let $T$ be a strict subset of $[n]$ such that
$A[[m],T]=B[[m],T]=0$.  Then
 $\rho(A,B)=\rho(A(\emptyset, T),B(\emptyset, T))$ and 
 $\hat\rho(A,B)=\hat\rho(A(\emptyset, T),B(\emptyset, T))$.
 \item
 Suppose that 
 for each zero $i$ row of $B$ the row $i$ of $A$ is zero.  Then  
 $r(A,B,\1_n)<\infty$.
\item There exists $\x\in\R^n_+\setminus\{\0\}$ such that $r(A,B,\x)<\infty$ if and only if there exists a nonempty subset $T\subseteq [n]$ such that $r(A,B,\1_T)<\infty$.
\item The weak Collatz-Wielandt quotient is positive if and only if the union of the supports of the rows of $A$ is $[n]$, i.e., $A\trans\1_m>\0$.
\item Assume that $A_1,B_1\in\R^{m\times n}_+$ and $A_1\le A, B\le B_1$ then $\rho(A_1,B_1)\le \rho(A,B)$ and $\hat\rho(A_1,B_1)\le \hat \rho(A,B)$.
 \end{enumerate}   
 \end{lemma} 
 
 The following lemma gives a lower bound on $\hat\rho(A,B)$:
 \begin{lemma}\label{lowestrhoAB}  Let $A=[a_{ij}],B=[b_{ij}]\in\R_{+}^{m\times n}$.
 Then
 \begin{equation}\label{lowestrhoAB1}
 \hat\rho(A,B)\ge \min_{j\in[n]} \frac{\sum_{i=1}^m a_{ij}}{\sum_{i=1}^m b_{ij}}.
 \end{equation}
 \end{lemma}
 \begin{proof} Clearly, it is enough to assume that $\hat\rho(A,B)<\infty$.  Assume that $\x\gneq \0$ and $r(A,B,\x)<\infty$.  Observe that
 \[r(A,B,\x)\ge \frac{\sum_{i=1}^m (A\x)_i}{\sum_{i=1}^m (B\x)_i}=\frac{\1_m A\x}{\1_m B\x}=\frac{\sum_{j=1}^n (\1_m A)_jx_j}{\sum_{j=1}^n (\1_m B)_jx_j}\ge \min_{j\in [n]} \frac{(\1_m A)_j}{(\1_m B)_j}.\]
 \end{proof}
\begin{lemma}\label{cindrABfinite}
Assume that $A,B\in\R^{m \times n}_+$.   Let $\cI\subseteq [m]$ be the set of the zero rows of $B$.  Denote by $T$ be the union of the the supports of the the rows $i\in\cI$ of $A$.  Then one of the following conditions holds
\begin{enumerate}
\item If $\cI=\emptyset$ then $\rho(A,B)\le r(A,B,\1_n)<\infty$.
\item If $T=[n]$ then $\hat\rho(A,B)=\infty$
\item  Assume that $\cI\ne \emptyset$ and $T$ is a strict subset of $[n]$,  Then 
$\hat\rho(A,B)=\hat\rho(A(\cI,T),B(\cI,T))$.
\end{enumerate}
\end{lemma}
\begin{proof}  {\it (1)}  If $\cI=\emptyset$ then $(B\1_n)_i>0$ for each $i$ and $r(A,B,\1_n)=\max_{i\in[m]}\frac{(A\1_n)_i}{(B\1_n)_i}<\infty$.

\noindent
{\it (2)}  Suppose that $T=[n]$.  Let $\x\gneq \0$.  Then there exists $i\in\cI$ such that $(A\x)_i>0$ and $(B\x)_i=0$.  Hence $r(A,B,\x)=\infty$ which yields that $\hat\rho(A,B)=\infty$.

\noindent
{\it (3)} Suppose that $\cI\ne \emptyset, T\subset [n]$. Let $\x\gneq \0$ and assume that $(\supp \x)\cap T\ne \emptyset$.  Then there exists $i\in\cI$ such that $(A\x)_i>0$.  As $(B\x)_i=0$ it follows that $r(A,B,\x)=\infty$.  Hence to determine $\hat\rho(A,B)$ is enough to consider $\inf\{r(A,B,\x), \x\gneq \0, \supp\x\subseteq [n]\setminus T\}$.  
Note that if $\supp\x\subseteq [n]\setminus T$ then $(A\x)_i=(B\x)_i=0$ for $i\in\cI$.
Therefore $\hat\rho(A,B)=\hat \rho(A(\cI,T),B(\cI,T))$.
\end{proof}

The above lemma gives rise to a polynomial time algorithm to check if $\hat\rho(A,B)$ is finite or infinite:
\begin{Algorithm}\label{alg}
Given $A=[a_{ij}],B=[b_{ij}]\in\R^{m\times n}_+$ set $S=[n]$, $\cI=\cJ\subseteq [m]$ the set of zero rows of $B$ and $T$ the union of the the supports of the the rows $i\in\cI$ of $A$;

\noindent
While $\cI\ne\emptyset$ and $T\ne S$

\noindent
\quad Replace 
$ A,B,S$ by $A(\cI,T),B(\cI,T),S\setminus T$;

\noindent 
\quad Replace  $\cI$ by the set  of zero rows of $B$;

\noindent 
\quad Replace $T$ by the union of the the supports of the rows $i\in\cI$ of $A$;

\noindent
\quad Replace
$\cJ$ by $\cI\cup\cJ$;

\noindent 
Else 

\noindent
$\quad$ If $\cI=\emptyset$ then $\hat\rho(A,B)\le r(A,B,1_S)<\infty$ and stop;

\noindent
$\quad$ If $T=S$ then $\hat\rho(A,B)=\infty$ and stop;

\end{Algorithm}
Denote by $\Pi_n\subset \R^n_+$ the set of probability vectors on $\R^n_+$.  Let $\Pi_n^{o}$ be the interior of $\Pi_n$, i.e., all probability vectors with positive coordinates.

We now discuss some properties of $\rho(A,B)$ and $\hat\rho(A,B)$.
\begin{lemma}\label{scproprho} Let $A,B\in\R^{m\times n}_+$.
\begin{enumerate}
\item The function $r(A,B,\x)$ is lower semicontinuous on $\R_+^{m\times n}\times \R_+^{m\times n}\times (\R_+^n\setminus\{\0\})$.
\item Let $T\subset[n]$ be a nonempty subset such that $B\1_T>\0$.  Assume that $\x\in \R_+^n,\sup \y=T$.  Suppose that the sequence  $\0<\x_k\in\R^n, k\in\N$  converges to $\x$.  Then $\lim_{k\to\infty}r(A,B,\x_k)=r(A,B,\x)$.  In particular $\rho(A[[m],T],B[[m],T])\ge \rho(A,B)$.
\item  Assume that
\begin{eqnarray*}
A=[A_1\;A_2], \;A_1=\left[\begin{array}{c}A_{11}\\0\end{array}\right], \; A_2=\left[\begin{array}{c}A_{12}\\A_{22}\end{array}\right],\;B=[B_1\;B_2]\;,B_1=\left[\begin{array}{c}B_{11}\\0\end{array}\right],\;B_2=\left[\begin{array}{c}B_{12}\\B_{22}\end{array}\right],\\
\end{eqnarray*}
where $A_{11},B_{11}\in \R^{m'\times \ell}$, $1\le m'<m, 1\le \ell<n$, and $B_{11}\1_{\ell}>\0$. Suppose that $\rho(A_{11},B_{11})<\rho(A,B)$.  Then $\rho(A_{22},B_{22})=\rho(A,B)$.
\end{enumerate}
\end{lemma}
\begin{proof} \emph{(1)} Suppose that $A_k,B_k\in\R_+^{m\times n}, \x_k\in \R_+^n\setminus\{\0\}, k\in\N$ and assume that
\begin{eqnarray*}
\lim_{k\to\infty} A_k=A, \quad \lim_{k\to\infty} B_k=B, \quad \lim_{k\to\infty} \x_k=\x\in\R_+^{n}\setminus\{\0\}.
\end{eqnarray*}
Suppose that $t=\liminf_{k\to\infty} r(A_k,B_k,\x_k)\in[0,\infty]$.  If $t=\infty$ then $r(A,B,\x)\le t$.  Assume that $t\in[0,\infty)$.  By passing to subsequences we can assume without loss of generality that $\lim_{k\to\infty} r(A_k,B_k,\x_k)=t$.  As $A_k\x_k\le r(A_k,B_k,\x_k)\x_k$ we deduce that $A\x\le tB\x$.  Hence $r(A,B,\x)\le t$.

\noindent
\emph{(2)} As $B\1_T>\0$ and $\sup \x=T$ we deduce that $B\x>\0$.  Hence 
$$\rho(A,B)\le\lim _{k\to\infty}r(A,B,\x_k)=\lim _{k\to\infty}\max_{i\in[m]}\frac{(A\x_k)_i}{(B\x_k)_i}=\max_{i\in[m]}\frac{(A\x)_i}{(B\x)_i}=r(A,B,\x).$$
Let $\z\in \R^{|T|}$ be the projection of $\x$ on its support.  Then $\z>\0$ and $r(A,B,\x)=r(A[[m],T],B[[m],T],\z)$.  Hence  $\rho(A[[m],T],B[[m],T])\ge \rho(A,B)$.

\noindent
\emph{(3)}  Assume that $\rho(A_{11},B_{11})<t_0=\rho(A,B)$.  Suppose to the contrary that $\rho(A_{22},B_{22})<t_0$.  Let $\0<\z\in \R^\ell, \0<\uu\in\R^{n-\ell}$ such that
$r(A_{11},B_{11},\z)<t_0, r(A_{22},B_{22},\uu)<t_0$.  For $s>0$ let $\x(s)=(\z\trans,s\uu\trans)\trans$. As $B_{11}\z>\0$ there exists $K>0$ such that $B_{11}\z\ge KA_{12}\uu$.  Then 
\begin{eqnarray*}
&&A_{11}\z+sA_{12}\uu\le r(A_{11},B_{11},\z)B_{11}\z+sKB_{11}\z=(r(A_{11},B_{11}+sK)B_{11}\z,\\
&&A_{22}s\uu\le r(A_{22},B_{22})s\uu.
\end{eqnarray*}
Thus for small enough $s$ we have the inequality $r(A,B,\x(s))<t_0$, which is a contradiction.
\end{proof}
 \begin{lemma}\label{rhoABachiev}  Let $m, n\ge 1$ be integers.  Assume that $A,B\in\R^{m\times n}_+$. Then
 \begin{enumerate}
 \item There exists $\y\in\R^n_+\setminus\{\0\}$ such that $\hat\rho(A,B)=r(A,B,\y)$.  (Such $\y$ is called a weak optimal $\y$.) 
 \item
 There exists $\y\in\R^n_+\setminus\{\0\}$ such that $\rho(A,B)=r(A,B,\y)$ with the following property:  There exist a sequence $\y_k>\0$ for $k\in\N$ such that $\lim_{k\to\infty}\y_k=\y$ and $\lim_{k\to\infty}r(A,B,\y_k)=\rho(A,B)=r(A,B,\y)$.   (Such $\y$ is called an optimal $\y$.)  
 \end{enumerate}
 \end{lemma}
 \begin{proof}  
 \emph{(1)}  If $\hat\rho(A,B)=\infty$ then each $\x\in\Pi_n$ is weak optimal.  Assume that $\hat\rho(A,B)<\infty$.
 Choose a sequence of $\y_k\in\Pi_n$, such that $t_k:=r(A,B,\y_k)\in (0,\infty), k\in\N$, $t_{k}\ge t_{k+1}$ for $k\in\N$, such that $\lim_{k\to\infty}=\hat\rho(A,B)$.  Pick up a subsequence of $\y_k$ which converges to $\y\in\Pi_n$.  
 Part \emph{(1)} of Lemma \ref{scproprho} yields that $r(A,B,\y)\le \hat\rho(A,B)$.  Thus  $r(A,B,\y)= \hat\rho(A,B)$ and $\y$ is weak optimal.
 
 \noindent
 \emph{(2)}  If $\rho(A,B)=\infty$ then each $\x\in\Pi_n^o$ is optimal.   Assume that $\rho(A,B)<\infty$.  Without loss of generality we may assume that $B$ does not have a zero row.
 We show by induction on $n$ that there exists an optimal $\y$.  For $n=1$ this claim is trivial.  Assume that the claim holds for $n\le N$. Suppose that $n=N+1$. 
 There exists a sequence of $\y_k\in\Pi_n^o$, such that $t_k:=r(A,B,\y_k)\in (0,\infty), k\in\N$, $t_{k}\ge t_{k+1}$ for $k\in\N$, such that $\lim_{k\to\infty}=\rho(A,B)$.   Pick up a subsequence of $\y_k$ which converges to $\w\in\Pi_n$.    Part \emph{(1)} of Lemma \ref{scproprho} yields that $r(A,B,\w)\le \rho(A,B)$. 
 If $\rho(A,B)=0$ we deduce that $\w$ is optimal.  Assume that $\rho(A,B)>0$.  If $r(A,B,\w)=\rho(A,B)$ then $\w$ is minimal.  Assume that $r(A,B,\w)<\rho(A,B)$.
 Hence $T=\supp\w$ is a strict subset of $[n]$.  Furthermore part \emph{(2)} of Lemma  \ref{scproprho} yields that $B\1_T$ is not positive. 
 
By relabeling the elements of $[n]$ we can assume that $T=[\ell]$ for some $\ell\in [n-1]$.   Let $\cK=\{\ell+1,\ldots,n\}$ and denote
 \[A_1=A[[m],[\ell]],\; B_1=B[[m],[\ell]],\; A_2=A[[m],\cK],\; B_2=B[[m],\cK].\] 
 Let $\cI$ be the set of zero rows of $B_1$.  As $B_1\1_\ell$ is not positive $|\cI|\ge 1$.
 As $r(A,B,\w)<\infty$ we deduce that $\cI$ is a subset of zero rows of $A_1$. Let $\0<\z\in \R^{\ell}$ be the projection of $\w$ on its support.  Relabel the rows of $A$ and $B$ such that $\cI=\{m'+1,\ldots,m\}$.
 
 Assume first that $m'=0$, i.e., $\cI=[m]$.  So $A_1=B_1=0$.  Part \emph{4} of  of Lemma \ref{rowB=0} yields that $\rho(A,B)=\rho(A_2,B_2)$.  As $|\cK|=n-\ell<n$ we can apply the induction hypothesis to $(A_2,B_2)$ to deduce the existence of an optimal $\uu\in\Pi_{n-\ell}, \supp \y\subseteq \cK$.  That is, there exists a sequence $\0<\uu_k\in\R^{n-\ell}, k\in\N$ such that $\lim_{k\to\infty} \uu_k=\uu$ and $\lim_{k\to\infty}r(A_2,B_2,\uu_k)=\rho(A_2,B_2)$.  Choose a sequence $\0<\z_k\in\R^{\ell}$ such that $\lim_{k\to\infty} \z_k=\0$.  Let $\vv_k=(\z_k\trans,\uu_k\trans)>\0$ for $k\in\N$.  Clearly $r(A,B,\vv_k)=r(A_2,B_2,\uu_k)$.  Hence $\y=(\0\trans, \uu\trans)\trans)$ is optimal.
 
  Assume that $m'=m-|\cI|\in [m-1]$.  It now follows that the conditions of part \emph{(3)} of Lemma  \ref{scproprho} holds.
The induction hypothesis yields that there exists $\uu\in\R_+^{ n-\ell}\setminus\{\0\}$ such that $r(A_{22},B_{22},\uu)=\rho(A_{22},B_{22})=\rho(A,B)$. 
Furthermore, there exists a sequence $\0<\vv_k\in \R^{n-\ell}, k\in\N$ such that $\lim_{k\to\infty}\vv_k=\vv$ and $\lim_{k\to\infty} r(A_2,B_2,\vv_k)=\rho(A_2,B_2)$.  For $s>0$ let $\vv(s)=(\z\trans, s\vv)\trans$.  As in the proof of  of part \emph{(3)} of Lemma  \ref{scproprho} we deduce $r(A,B,\vv(s))\le \max(r(A_{11},B_{11},\z)+Ks, \rho(A,B))$.  Choose $s_0>0$ such that $r(A_{11},B_{11},\z)+Ks_0<\rho(A,B)$.   Then $r(A,B,\vv(s_0))=\rho(A_2,B_2)=\rho(A,B)$.   Set $\vv_k=(\z\trans,s_0\uu_k\trans)\trans >\0, k\in\N$.  So  $\lim_{k\to\infty}\vv_k=\vv(t_0)$.
Choose $K'>K$ such that $r(A_{11},B_{11},\z)+K's_0< \rho(A,B)$.  As $lim_{k\to\infty}\uu_k\to \uu$ it follows that there exist $N$ such that $A_{12}\uu_k\le K'B\z$  for $k>N$.  Therefore $r(A,B,\vv_k)=\rho(A_2,B_2,\uu_k)$ for $k>N$.  This implies that $\vv(s_0)$ is optimal.
\end{proof}
 
 We call $\y$ a \emph{minimal} (weak) optimal if $\y$ is (weak) optimal and there is no (weak) optimal $\w$ such that the support of $\w$ is strictly contained in the support of $\y$.  A vector $\y\gneq \0$ is called a weak GPF-eigenvector if $A\y=\hat\rho(A,B)\y$.

 The next lemma discusses connections between $\rho(A,B)$ and $\hat\rho(A,B)$.
 \begin{lemma}\label{Bpos}Let $A,B\in\R_+^{m,n}$.   Then 
 \begin{enumerate}
 \item $\hat\rho(A,B)\le \rho(A,B)$.
 \item Let $\y$ be weak optimal.  If either $A\y>\0$ or $B\y>\0$ then  $\rho(A,B)=\hat\rho(A,B)$ and $\y$ is optimal.
\item Assume that either $A>0$ or $B>0$.  Then $\rho(A,B)=\hat\rho(A,B)$ and each weak optimal $\y$ is optimal.
\item Assume that $0<B_l\in\R^{m\times n}$ for $l\in \N$ and $\lim_{l\to\infty} B_l=B$.
Then 
$\lim_{l\to\infty}\rho(A,B_l)=\hat\rho(A,B)$.
\item Assume that $0<A_l\in\R^{m\times n}$ for $l\in \N$ and $\lim_{l\to\infty} A_l=A$.
If $B$ does not have a zero row then 
$\lim_{l\to\infty}\rho(A_l,B)=\rho(A,B)$.
\end{enumerate}
\end{lemma}
\begin{proof}  \emph{(1)} The inequality $\hat\rho(A,B)\le \rho(A,B)$ is clear. 

\noindent
\emph{(2)}    Let $\y\in\Pi_n$ be weak optimal.  So $r(A,B,\y)=\hat\rho(A,B)$.
Assume first that $B\y>0$.  
For $t>0$ define $\y(t)=\y+t\1_n$.  Note that $\lim_{t\searrow 0} \y(t)=\y$.  As $(B\y)_i>0$ for $i\in[n]$ it follows that  $\lim_{t\searrow 0} \frac{(A\y(t))_i}{(B\y(t))_i}=\frac{(A\y)_i}{(B\y)_i}$.  Hence 
$$\rho(A,B)\le \lim _{t\searrow 0}r(A,B,\y(t))=r(A,B,\y)=\hat\rho(A,B)\Rightarrow \rho(A,B)=\hat\rho(A,B),$$
and $\y$ is optimal.

Assume second that $A\y>0$.  Suppose that $B$ has a zero row.  Then $\rho(A,B)=\hat\rho(A,B)=\infty$ and $\y$ is optimal.  Assume now that $B$ does not have a zero row. Then $\rho(A,B)<\infty$.  As $A\y\le \hat\rho(A,B)B\y$ it follows that $\hat\rho(A,B)>0$ and $B\y>\0$.  Therefore the first case yields $\rho(A,B)=\hat\rho(A,B)$ and $\y$ is optimal.

\noindent
\emph{(3)}  This claim follows from part \emph{2}.

\noindent
\emph{(4)}  Assume that we have a sequence $\x_l\in\Pi_n$ such that $r(A,B_l)=r(A,B_l,\x_l)$ for each $l\in\N$.   Denote $t=\limsup r(A,B_l), \; s=\liminf r(A,B_l)$.  
 Let $B_l=C_l+D_l, C_l,D_l\ge 0$, where $\supp C_l=\supp B$ and $\supp C_l\cap\supp D_l=\emptyset$.   Clearly, $\lim_{l\to\infty}C_l=B$ and $\lim_{l\to\infty} D_l=0$.  Fix $\varepsilon>0$. Then there exists $M(\varepsilon)$ such that $B\le (1+\varepsilon)C_l$ for $l>M(\varepsilon)$.  
 In particular, $B\le (1+\varepsilon)B_l$ for $l>M(\varepsilon)$.  Hence
 \[\frac{1}{1+\varepsilon}r(A,B_l,\x) =r(A,(1+\varepsilon)B_l,\x)\le r(A,B,\x).\]  
 Therefore
 \[\frac{1}{1+\varepsilon}\rho(A,B_l)=\frac{1}{1+\varepsilon}\hat\rho(A,B_l)=\hat\rho(A,(1+\varepsilon)B_l)\le \hat\rho(A,B) \textrm{ for } l>M(\varepsilon).\]
 As $\varepsilon>0$ was chose arbitrary it follows that $t\le \hat\rho(A,B)$.
 
 There exists a subsequence $1\le l_1<l_2<\cdots$ such that $\lim_{k\to\infty} \rho(A,B_{l_k})=s$ and $\lim_{k\to\infty} \x_{l_k}=\x\in\Pi_n$.  The inequality
 $A\x_{l_k}\le r(A,B_{l_k})B_{l_k}\x_{l_k}$ yield $A\x\le sB\x$.  (We may have that $s=\infty$.)  Suppose first that $s<\infty$. Hence $s\ge \hat\rho(A,B)$. 
 Combine that with the inequality $t\le \hat\rho(A,B)$ to deduce that $s=t=\hat\rho(A,B)$.
 Assume second that $s=\infty$.  Then $t=\infty$ and the inequality $t\le \hat\rho(A,B)$
 yields that $\hat\rho(A,B)=\infty$. 
 
 \noindent
 \emph{(5)}   Denote $t=\limsup r(A_l,B), \; s=\liminf r(A_l,B)$.   Let $A_l=C_l+D_l, C_l,D_l\ge 0$, where $\supp C_l=\supp A$ and $\supp C_l\cap\supp D_l=\emptyset$.   Clearly, $\lim_{l\to\infty}C_l=A$ and $\lim_{l\to\infty} D_l=0$.  Fix $\varepsilon\in (0,1)$ in the rest of the proof. Then there exists $M(\varepsilon)$ such that $(1-\varepsilon)C_l\le A\le (1+\varepsilon)C_l$ for $l>M(\varepsilon)$.  
 In particular, $A\le (1+\varepsilon)A_l$ for $l>M(\varepsilon)$.  Hence for $\x\gneq \0$
 \[(1+\varepsilon)r(A_l,B,\x) =r((1+\varepsilon)A_l,B,\x)\ge r(A,B,\x).\]  
 Therefore
 \[(1+\varepsilon)\rho(A_l,B)=\rho((1+\varepsilon)A_l,B)\ge \rho(A,B) \textrm{ for } l>M(\varepsilon).\]
 As $\varepsilon>0$ was chosen arbitrary it follows that $s\ge \rho(A,B)$.
 As $B$ does not have zero row it follows that $\rho(A,B)\le r(A,B,\1_n)<\infty$.  Let $\x(\varepsilon)\in\Pi_n^o$ satisfy $A\x(\varepsilon)\le (\rho(A,B)+\varepsilon)B\x(\varepsilon)$.  As $B$ does not have a zero row it follows that $B\x(\varepsilon)>\0$.  Since $\lim_{l\to\infty} D_l=0$  there exists $L(\varepsilon)\ge M(\varepsilon)$ such
 that $D_l\x(\varepsilon)\le \varepsilon B\x(\varepsilon)$ for $l> L(\varepsilon)$.  Then
 \begin{eqnarray*}
 &&A_l\x(\varepsilon)=C_l\x(\varepsilon)+D_l\x(\varepsilon)\le (1-\varepsilon)^{-1} A\x(\varepsilon)+\varepsilon B\x(\varepsilon)\le\\
 &&((1-\varepsilon)^{-1}(\rho(A,B)+\varepsilon)+\varepsilon)B\x(\varepsilon) \textrm{ for } l> L(\varepsilon).
 \end{eqnarray*}
 That is, $\rho(A_l,B)\le ((1-\varepsilon)^{-1}(\rho(A,B)+\varepsilon)$ for $l>L(\varepsilon)$.  Therefore $t\le  ((1-\varepsilon)^{-1}(\rho(A,B)+\varepsilon)+\varepsilon)$.  As $\varepsilon$ was an arbitrary number in the open interval $(0,1)$ it follows that $t\le \rho(A,B)$.  
 Combine that with the inequality $s\ge \rho(A,B)$ to deduce that $s=t=\rho(A,B)$.
 \end{proof}
 \section{The classical case}\label{sec:clas}
 In this section we discuss the case where $m=n$ and $B$ is the identity matrix $I$.
 We first recall some basic results on directed graphs $ \vec{G}=(V,\vec{E})$.
 Here $V$ is a finite set of vertices and $\vec{E}\subseteq V\times V$ is the set of diedges of $\vec{G}$.  An ordered tuple $(v,w)\in \vec{E}$ is a diedge from $v$ to $w$.
 The diedge $(v,v)$ is called a loop.  A dipath $\vec{P}$ in $\vec{G}$ is an ordered set of diedges $\{(v_1,v_{2}), (v_2,v_3),\ldots,(v_p,v_{p+1})\}$ for a positive integer $p$.  
 A dipath $\vec{P}$ is closed if $v_{p+1}=v_1$.  The digraph $\vec{G}$ is called acyclic or diforest if $\vec{G}$ does not have a closed path.
 The digraph $\vec{G}$ is called strongly connected if for any two distinct vertices $v,w\in V$ there is a dipath in $\vec{G}$ from $v$ to $w$. (A digraph on one vertex with no diedge is strongly connected.)
 
  For $W\subseteq V$ we define the induced subdigraph $\vec{G}(W)=(W,\vec{E}(W))$, where $\vec{E}(W)$ is the set of diedges in $\vec{E}$ that connects two vertices in $W$.  Assume that $\vec{G}$ is not strongly connected.  Then the subgraph $\vec{G}(W)$ is called strongly connected component of $\vec{G}$ if the subgraph $\vec{G}(W)$ is strongly connected but $\vec{G}(U)$  is not strongly connected for each $U$ that strictly contains $W$.  Let $V=\cup_{i=1}^k V_i$ be the partition of $V$ corresponding to the strongly connected components of $\vec{G}$.  That is $\vec{G}(V_i)$ for $i\in[k]$ are all strongly connected components of $\vec{G}$.  
  The reduced digraph $\vec{G}_{red}=(V_{red},\vec{E}_{red})$ of $\vec{G}$ is defined as follows:  First $V_{red}=\{\{V_1\},\ldots,,\{V_k\}\}$.  Second a diedge $(\{V_i\},\{V_j\})$ is in $\vec{E}_{red}$ if $i\ne j$ and there is a diedge in $\vec{E}$ from $V_i$ to $V_j$.  For a strongly connected digraph $\vec{G}$ we let $V_{red}=\{\{V\}\}$ and $\vec{E}_{red}=\emptyset$.  It is straightforward to show that $\vec{G}_{red}$ is acyclic.  With each digraph we associated an undirected graph $G=(V,E)$, where undirected edge $\{i,j\}$ is in $E$ if either $(i,j)$ or $(j,i)$ in $\vec{E}$.  Then $G$ is a union of its connected components $G(W_j), j\in[c]$, where each $V_i$ is a subset of some $W_j$.  Clearly, each $W_j$ is union of some subsets $V_1,\ldots,V_k$.  The subset $W_j$ induces a subdigraph $\vec{G}(W_j)$.  Observe that $\vec{G}(W_j)$ induces a reduced digraph $\vec{G}(W_j)_{red}$, which is a subdigraph of the reduced graph of $\vec{G}_{red}$.  The subdigraph $\vec{G}(W_j)_{red}$ is called a ditree of $\vec{G}_{red}$.  A vertex $\{V_p\}$ is called a \emph{source} of $\vec{G}_{red}$ if there is no $\{V_q\}$ such that there is a diedge in $\vec{G}_{red}$ from $\{V_q\}$ to $\{V_p\}$.  A vertex $\{V_p\}$ is called a \emph{sink} of $\vec{G}_{red}$ if there is no $\{V_q\}$ such that there is a diedge in $\vec{G}_{red}$ from $\{V_p\}$ to $\{V_q\}$.
  Clearly, each ditree $\vec{G}(W_j)_{red}$ contains at least one source $\{V_p\}$ and one sink $\{V_q\}$.  Denote by $\cR(\vec{G})\subseteq [k]$ all indices $p\in[k]$ such that $\{V_p\}$ is a source.
  
Let $A=[a_{ij}]\in\R_+^{m\times m}$.  One associates with $A$ the digraph $\vec{G}(A)=([m], \vec{E}(A))$.   A diedge $(i,j)$ is in $\vec{E}(A)$ if and only if $a_{ij}>0$.  
$A$ is called \emph{irreducible} if $\vec{G}(A)$ is strongly connected.  Let $G(A)$ be the induced undirected graph.  Then after renaming the indices $A$ is  pemutationally similar to a block diagonal matrix $\diag(A_1,\ldots,A_c)$ if and only if $G(A)$ has $c$-connected components. Each $A_j$ corresponds to $\vec{G}(W_j)$.
For each nonzero subset $U \subseteq [m]$ denote by $A[U]$ the submatrix $[a_{ij}]_{i,j\in U}$.  Each $A_j$ is can be assumed to be in the Frobenius normal form \cite[Theorem 6.4.4]{Fri15}.  It is a block upper triangular form, where each diagonal block is an irreducible matrix $A[V_i]$, where $V_i\subseteq W_j$.  (The order of the diagonal blocks depends on the labeling of the vertices of the ditree $\vec{G}(W_j)_{red}$. One such labeling is given in \cite[Theorem 6.4.4]{Fri15}.)  We assume that the top diagonal block of $A[W_j]$ is $A[V_p]$ where $\{V_p\}$ is a source in the ditree $\vec{G}(W_j)_{red}$ and the bottom diagonal block is $A[V_q]$ where $\{V_q\}$ is a sink in $\vec{G}(W_j)_{red}$.  Furthermore, every source $\{V_p\}$ in can be chosen to be  the top diagonal block in $A[W_j]$.  In particular 
\[\rho(A)=\max_{i\in[k]}\rho(A[V_i]), \quad \rho(A[W_j])=\max_{V_i\subseteq W_j}\rho(A[V_i]).\]

We first bring the well known result due to Wielandt \cite{Wie50}.
\begin{lemma}\label{Wiel}  Let $A\in\R^{m\times m}_+$ be an irreducible matrix.  
Then 
$\rho(A,I)=\hat\rho(A,I)=\rho(A)$.
Furthermore, in characterization \eqref{ColWiel} equality holds if and only if $\x$ is the PF-eigenvector of $A$.
\end{lemma}
\begin{proof}  The equality $\rho(A,I)=\rho(A)$, i.e., the characterization \eqref{ColWiel} was proved by Wielandt \cite{Wie50}.  Wielandt also showed that  equality in \eqref{ColWiel} holds  if and only if $\x$ is the PF-eigenvector of $A$.
The equality $\rho(A,I)=\hat\rho(A,I)$ follows from the following observation:  Assume that $\x=(x_1,\ldots,x_m)\trans\gneq 0$ and $1\le |\supp \x|<m$.  Then $r(A,I,\x)=\infty$.  Indeed, let $\cI=[m]\setminus \supp \x$.  As $\vec{G}(A)$ is irreducible there exists $(i,j)\in \vec{E}(A)$ such that $i\in\cI$ and $j\in \supp \x$.  Hence $a_{ij}>0$.  Therefore $x_i=0$ and $(A\x)_i>0$, which yield that $\frac{(A\x)_i}{x_i}=\infty$.
\end{proof}
\begin{theorem}\label{charrhoAI}  Assume that $A\in\R^{m\times m}_+$.  Then
\begin{enumerate}
\item
$\rho(A,I)=\rho(A)$, and  $A\y=\rho(A)\y$ for some $\y\gneq \0$. 
\item There exists $\x=(x_1,\ldots,x_m)\trans >\0$ such that $\rho(A)=\max_{i\in[m]} \frac{(A\x)_i}{x_i}$ if and only if the following condition hold:
 Let $\vec{G}_{red}$ be the induced reduced graph by $A$, with the set of vertices $\{\{V_1\},\ldots,\{V_k\}\}$.  The equality $\rho(A[V_j])=\rho(A)$ implies that $\{V_j\}$ is a sink of $\vec{G}_{red}$. 
\item The vector $\x\in\Pi_m^o$ such that $\rho(A)=\max_{i\in[m]} \frac{(A\x)_i}{x_i}$ is unique if and only if $A$ is irreducible.
\item $\hat\rho(A,I)=\min_{i\in\cR(\vec{G}(A))}\rho(A[V_i])$ and  $A\w=\hat\rho(A,I)\w$  for some  $\w\gneq \0$.
\end{enumerate}
\end{theorem}
 \begin{proof}  \emph{(1)}  Recall the Perron's result that claims that a positive square matrix $C$ has a positive eigenvector $\uu$ corresponding to $\rho(C)$.  Assume that $B\in\R_+^{m\times m}$.  Let $J\in\R^{m\times m}_+$ be a matrix whose all entries are $1$.  Set $B_l=B+\frac{1}{l}J$ for each positive integer $l$.  Perron's theorem implies the existence of $\uu_l\in\Pi_m^{o}$ such that $B_l\uu_l=\rho(B_l)\uu_l$.  Since $ \Pi_m$ is compact there is a subsequence of $\{\uu_l\}$ which converges to $\uu\in \Pi_m$.  Clearly, $\lim_{l\to\infty}\rho(B_l)=\rho(B)$.  Hence $B\uu=\rho(B)\uu$.  Choose $B=A$.  Then $r(A,I,\uu_l)<\rho(A_l,I.\uu_l)=\rho(A_l)$.  Therefore  $\rho(A,I)\le \rho(A)$.  Furthermore $A$ has a nonnegative eigenvector $\y$ corresponding to $A$.  Let $\x>\0$.  We show that $r(A,I,\x)\ge \rho(A)$.  Assume that $A\trans \uu=\rho(A)\uu$, where $\uu\gneq \0$.  Then 
 $\uu\trans(r(A,I,\x) \x)\ge \uu\trans A\x=\rho(A)\uu\trans \x$.  As $\x>\0$ we deduce that $\uu\trans \x>0$.  Hence $r(A,I)\ge \rho(A)$.   (This argument is in \cite{Col42}.)
 Hence $\rho(A,I)\ge \rho(A)$ and $\rho(A,I)=\rho(A)$.
 
 \noindent
 \emph{(2)}
 Assume that there exists $\x=(x_1,\ldots,x_m)\trans >\0$ such that
 
 \noindent
$\rho(A)=\max_{i\in[m]} \frac{(A\x)_i}{x_i}$.  Let $\{\{V_1\},\ldots,\{V_k\}\}$ be the vertices of the reduced graph induced by $A$.  Suppose that $\rho(A)=\rho(A[V_j])$.  Let $\z=(z_1,\ldots,z_l)\trans\in \R_+^{|V_j|}$ be the subvector of $\x$ restricted to the set $V_j$.  Hence 
\[\rho(A)=\rho(A[V_j])\le \max_{i\in [|V_j|]}\frac{(A[V_j]\z)_i}{z_i}\le \max_{i\in V_j} \frac{(A\x)_i}{x_i}\le \rho(A).\]  
Therefore all the inequalities are equalities.  The first equality and 
Lemma \ref{Wiel} yield that $A[V_j]\z=\rho(A[V_j])\z$.  The second equality yields that $V_j$ is a sink.  

Vice versa assume that  $\rho(A[V_j])=\rho(A)$ if and only if $\{V_j\}$ is a sink of $\vec{G}_{red}$. 
We now show that there exists $\x=(x_1,\ldots,x_m)\trans >\0$ such that equality holds in \eqref{ColWiel}. Without loss of generality we may assume that $\rho(A)>0$.  (Otherwise  $A=0$.)   
Let $\cC_0(\vec{G}_{red})\subseteq\{\{V_1\},\ldots,\{V_k\}\}$ be the set of the sink vertices in $\vec{G}_{red}$. 
Assume that $\{V_j\}$ is a sink.  Then the restriction of $\x$ to $V_j$ is the PF-eigenvector $\x_j>\0$ of $A[V_j]$.  If $\cC_0(\vec{G}_{red})=\{\{V_1\},\ldots,\{V_k\}\}$
we easily deduce that equality holds in \eqref{ColWiel} for this $\x$.  If not let us consider the subsets   $\cC_l(\vec{G}_{red})\subset\{\{V_1\},\ldots,\{V_k\}\}$ for $l=0,1,\ldots,p$, which is a partition of $\{\{V_1\},\ldots,\{V_k\}\}$, defined as follows.  For $l\ge 1$ the set $\cC_l(\vec{G}_{red})\subset\{\{V_1\},\ldots,\{V_k\}\}$ consists of vertices in $\vec{G}_{red}$ with the maximal path length $l$ in the digraph $\vec{G}_{red}$ starting at these vertices.
Thus the diedges in $\vec{G}_{red}$ from $\cC_l(\vec{G}_{red})$ go only to $\cC_{r}(\vec{G}_{red})$ for $r=0,1,\ldots,l-1$, and  each $\{V_j\}\in \cC_l(\vec{G}_{red})$ has at least one diedge.  (This corresponds to the Frobenius normal form given in  \cite[Theorem 6.4.4]{Fri15}.)  
Suppose that we already determined the restriction of $\x$ to $\{V_q\}\in \cC_{r}(\vec{G}_{red})$ for $r\le l-1$, which is denoted by $\x_q$, where $\x_q>0$. We now show how to determine $\x_j$ for each $\{V_j\}\in \cC_{l}(\vec{G}_{red})$.
 Recall that $\rho(A[V_j])<\rho(A)$.  Choose $t_j\in(\rho(A[V_j]),\rho(A)]$.  We now determine $\x_j$ by the condition $(A\x)[V_j]=t_j\x_j$.  Denote by $A[V_j,V_q]$ the restriction of $A$ to rows in the set $V_j$ and columns in the set in $V_q$.  Then the above condition is equivalent to
 \[(t_jI_{V_j}-A[V_j])\x_j=\sum_{V_q\in\cup_{r=0}^{l-1} \cC_{r}(\vec{G}_{red})}A[V_j,V_q]\x_q.\]
 Since all $\x_q>\0$ for $V_q\in\cup_{r=0}^{l-1} \cC_{r}(\vec{G}_{red})$ and at least one of $A[V_j,V_q]$ is a nonzero nonnegative matrix it follows that the right  hand side in the above equality is a nonzero nonnegative vector.  As $t_j>\rho(A(V_j))$ and $A[V_j]$ is irreducible it follows that $(t_jI_{V_j}-A[V_j])^{-1}$ is a positive matrix \cite[Lemma 6.4.3]{Fri15}.  Therefore 
  \[\x_j=(t_jI_{V_j}-A(V_j))^{-1}\sum_{V_q\in\cup_{r=0}^{l-1} \cC_{r}(\vec{G}_{red})}A[V_j,V_q]\x_q>\0.\] 
This shows that the constructed $\x$ is a positive vector.  It is left to show that $\rho(A)=\max_{i\in[m]} \frac{(A\x)_i}{x_i}$.  Assume that $i\in V_j$.  Then our construction gives that $\frac{(A\x)_i}{x_i}=t_j\le \rho(A)$.  For a sink $\{V_j\}$ such that $\rho(A(V_j))=\rho(A)$ we have that $t_j=\rho(A)$.  Hence $\rho(A)=\max_{i\in[m]} \frac{(A\x)_i}{x_i}$.
 
 \noindent
 \emph{(3)}  The proof of \emph{2} shows that $\x\in\Pi_m^{o}$ is unique if and only if 
 $A$ is an irreducible matrix.
 
 \noindent
 \emph{(4)} Clearly, it is enough to show the equality $\hat\rho(A,I)=\min_{i\in\cR(\vec{G}(A))}\rho(A[V_i])$  in the case where $\vec{G}(A)_{red}$ is a ditree, that is, $G(A)$ is a connected graph.  Let $\{V_i\}$ be a source of $\vec{G}(A)_{red}$.  So we can choose a Frobenius normal form so that the irreducible matrix $A[V_i]$ appears in the first diagonal block of the Frobenius normal form $F$ of $A$.  Let $A[V_i]\z=\rho(A[V_i])\z$ where $\z>0$.  Extend $\z$ to $\R_+^m$ by addind zero entries for indices $\{|V_i|+1,\ldots,m\}$ to obtain the vector $\vv\in \R^m_+$.  Then $F\vv=\rho(A[U_i])\vv$.  
 Rename the name of the indices to deduce that $A\w=\rho(A[U_i])\w$.  Hence $r(A,I,\w)=\rho(A[U_i])$ and $\hat\rho(A,I)\le \rho(A[U_i])$.  This shows that $\hat\rho(A,I)\le \alpha=\min_{i\in\cR(\vec{G}(A))}\rho(A[V_i])$. 
 
 It is left to show the reverse inequality  $r(A,I,\x)\ge \alpha$ for each $\x\gneq \0$ such that $r(A,I,\x)<\infty$.
 Suppose first that $\x>\0$.  The above arguments yield that $r(A,I,\x)\ge \rho(A)\ge \alpha$.  Assume now that $\supp \x$ is a strict subset of $[m]$.   Let $\cI\subseteq [k]$ the the set of all $i\in[k]$ such that $\supp \x\cap V_i\ne \emptyset$.  We claim that for each $i\in\cI$ we have the equality $\supp \x\cap V_i= V_i$.  Indeed assume that $\supp \x \cap V_i$ is a strict subset of $V_i$.  Let $\x_i=\x[V_i]\in\R_+^{|V_i|}$ be the restriction of $\x$ to $V_i$. Thus $\x_i\gneq 0$.  Since $A[V_i]$ is an irreducible matrix, the proof of Lemma \ref{Wiel} yield that there exists $p\in V_i\setminus\supp \x$ such that $(A[V_i]\x_i)_p>0$.  As $(A\x)_p\ge (A[V_i]\x_i)_p$ it follows that $\frac{(A\x)_p}{x_p}=\frac{+}{0}=\infty$.
 So $r(A,I,\x)=\infty$ contrary to our assumption. Assume that $\{V_i\}$ is a source in $\vec{G}(A)_{red}$.  Clearly  $r(A,I,\x)\ge r(A[V_i],I, \x_i)$.
 Lemma \ref{Wiel} yields that $ r(A[V_i],I, \x_i)\ge \rho(A[V_i])$.  Hence  $r(A[V_i],I, \x_i)\ge \alpha$.  Assume that $\{V_i\}$ is not a source.  Hence there is $j\in [k]$ such that $(j,i)\in \vec{E}(\vec{G}(A)_{red})$.  We claim that $\supp \x\cap V_j=V_j$.  Suppose not.  So $\supp \x\cap V_j=\emptyset$.  From the definition of the reduced graph $\vec{G}(A)_{red}$ it follows there is $p\in V_j$ and $q\in V_i$ such that $a_{pq}>0$.  Hence $(A\x)_p>0$ and $x_p=0$.  Thus $r(A,I,\x)=\infty$ whi‎ch contradicts our assumption.  Thus $\supp \x\cap V_j=V_j$.  Repeating this argument a number of steps we deduce that there is a source vertex $\{V_l\}$ in $\vec{G}(A)_{red}$ such that $\supp\x\cap V_l=V_l$.  The previous arguments yield that $r(A,I,\x)\ge \rho (A[V_l])\ge \alpha$.
 \end{proof}
 
 We remark that part \emph{1} of this theorem is well known, part \emph{2} is close to \cite[Theorem 6.4.5]{Fri15}, part \emph{3} is perhaps known, and part \emph{4} seems to be new.
 
 \begin{example} Let $A=\left[\begin{array}{ccc}a&1&1\\0&1&1\\0&0&1+a \end{array}\right]$, where $a\in(0,1]$, and $B=I_3$.  Then 
 \begin{enumerate}
 \item $\rho(A,I_3)=\rho(A)=1+a$, $\x=(1+2a^{-1},2a^{-1},1)\trans$ is optimal and $\y=(1,1,0)\trans$ is minimal optimal.
 \item $\hat\rho(A,I)=a$ and $\y=(1,0,0)\trans$ is the unique weak minimal vector in $\Pi_3$.
 \end{enumerate}
 \end{example}
 \begin{proof} \emph{(1)} It is straightford to show that  $r(A,I,\x)=(1+a)$, hence $\x$ is optimal.  Take $\y(t)=(1,1,t)\trans$ for $t>0$.  Then $r(A,B,\y(t))=1+a+t$. Let $t\to 0$ to deduce that $\y$ is optimal.  Suppose to the contrary that $\y$ is not minimal optimal.  Thus either $\uu=(1,0,0)\trans$ or $\vv=(0,1,0)\trans$  is minimal optimal. 
 But $r(A,I,\uu)=a$ and $r(A,I,\vv)=\infty$.  Contradiction.
 
 \noindent
 \emph{(2)}  Part \emph{4} of Theorem \ref{charrhoAI} yields that $\hat\rho(A,I)=a$.  It is straightforward to show that  $\y=(1,0,0)\trans$ is the unique weak minimal vector in $\Pi_3$.
 \end{proof}
 \section{Polynomial approximation of $\rho(A,B)$ and $\hat\rho(A,B)$}\label{sec:polaprox}
 We first recall the fundamental result that a feasibility of a system of linear inequalities and minimization of a linear function are polynomial.  For simplicity we state a variant of this fact in the following setting that we need.
\begin{lemma}\label{polsolvlinineq} Let $C\in\Q^{m\times n},\mathbf{b}=(b_1,\ldots,b_m)\trans\in\Q^m, \mathbf{c}=(c_1,\ldots,c_n)\trans \in \Q^n$.  Then one can find in polynomial time in $\langle C\rangle+\langle\mathbf{b}\rangle$ if the the following polytope is empty or not 
\begin{equation}\label{llinineqPin}
C\x\ge \mathbf{b}, \quad\x=(x_1,\ldots,x_n)\trans\ge \0, \sum_{i=1}^n x_i=1.
\end{equation}
Furthermore, if the above polytope is nonempty then the minimum of the linear function $\mathbf{c}\trans \x$ over this polytope can be found in polynomial time.
\end{lemma}
This result is well known, see for example \cite{Lov86, GLS88}.  To be precise, \cite[Corollary 2.3.7]{Lov86} assumes for simplicity that the system \eqref{llinineqPin} is solvable, and then one can find the maximum or minimum of $\mathbf{c}\trans\x$, where $\mathbf{c}\in\Q^n$ is given and $\x$ satisfies   \eqref{llinineqPin}.   To apply \cite[Corollary 2.3.7]{Lov86} for solvable system we consider  the following linear programing problems $LP_j$ for $j\in [m]$.  Let $\mathbf{c}_j\trans$ be the $j-th$ row of $C$.  Let $C_j\in\Q^{(j-1)\times n}, \mathbf{b}_j\in\Q^{j-1}$ be the matrix and the column obtained from $C$ and $\mathbf{b}$ by deleting the $m-j+1$ rows $j,\ldots,m$ respectively.  Assume that the system
\begin{equation}\label{LPjprob}
C_j\x\ge \mathbf{b}_j,  \quad\x=(x_1,\ldots,x_n)\trans\ge \0, \sum_{i=1}^n x_i=1
\end{equation}
is solvable for $j+1=\ell\in[n]$.   (This is trivially true for $\ell=1$.) Now consider the maximum problem
$\max\mathbf{c}_{\ell}\trans\x$ over the set given by \eqref{LPjprob} for $j=\ell$.  Assume that the maximum is achieved at $\x_{\ell}$.  Then the polytope given by \eqref{LPjprob} for $j=\ell+1$ is nonempty if and only if $\mathbf{c}\trans \x_{\ell}\ge b_{\ell}$.
Hence by running at most $m$ linear programming problems we can determine in polynomial time if the system \eqref{llinineqPin} is feasible or not.
\begin{theorem}\label{epsapproxCWn}  Let $A,B\in\Q_+^{m\times n}\setminus\{0\}$.  Assume that $\rho(A,B)<\infty$.  Then for any $\varepsilon\in (0,1)\cap \Q$  one of the following statements can be verified in polynomial time in $\langle A\rangle+\langle B\rangle+\langle \varepsilon\rangle$:
\begin{enumerate}
\item The Collatz-Wielandt quotient satisfies $\rho(A,B)<\varepsilon$.
\item The Collatz-Wielandt quotient is positive and one can find
 $\tilde\rho(A,B)\in \Q_{+}\setminus \{0\}$, such that 
\begin{equation}\label{epsapproxCWn1}
\tilde \rho(A,B)\le \rho(A,B)\le (1+\varepsilon) \tilde\rho(A,B).
\end{equation}
\end{enumerate}
\end{theorem}
\begin{proof}  Let $t_0=r(A,B,\1_n)<\infty$ and $N=\lceil\log_2\varepsilon^{-1}\rceil +1$.  Set $k=1$, $t_{k}=\frac{1}{2}t_{k-1}$ and $C=t_k B-A$.  Consider the system 
\eqref{llinineqPin} with $\mathbf{b}=\0$.  Assume first that this system is solvable.
Let $\mu_{i,k}$ be the maximum of $x_i$ for the system \eqref{llinineqPin} with  $\mathbf{b}=\0$ for $i\in[n]$.   Assume first that $\mu_{i,k}>0$ for each $i\in[n]$.  We claim that $\rho(A,B)\le t_k$.  Indeed, assume that $\mu_{i,k}=x_{i,i,k}$ where $\x_{i,k}=(x_{1,i,k},\ldots,x_{n,i,k})\trans\in \Q^{n}$ satisfies the system 
\eqref{llinineqPin} with $\mathbf{b}=\0$.   Set $\x=\frac{1}{n}\sum_{i=1}^n \x_{i,k}$ to deduce that $r(A,B,\x)\le t_k$.  Assume that for $k=2,\ldots,N$ we have the inequality
$\rho(A,B)\le t_k$.  Then $\rho(A,B)\le t_N<\varepsilon$ and we showed the condition \emph{(1)}. 

Suppose now that for the smallest value $k\in[N]$ one of the following conditions hold: Either the system  \eqref{llinineqPin} with $C=t_k B-A$ and $\mathbf{b}=\0$  is not solvable or $\mu_{i,k}=0$ for some $i\in[n]$.  Then $0<t_k\le \rho(A,B)$.  Set $l=0$ $f_l=t_k,g_l=t_{k-1}$.  Then $f_l\le \rho(A,B)\le g_l$.  If $\frac{g_l-f_l}{f_l}\le \varepsilon$ then $\tilde \rho(A,B)=f_l$.  If not set $M=\lceil\log_{4/3}\frac{g_l-f_0}{f_0\varepsilon}\rceil$ and apply now the bisection algorithm:
Let $h_l=\frac{f_l+g_l}{2}$ and $C=h_lB-A$.  Consider the system \eqref{llinineqPin} with $\mathbf{b}=\0$.  Assume first that this system is solvable.  Let $\mu_{i,l}$ be the minimum of $x_i$ for the system \eqref{llinineqPin} with $\mathbf{b}=\0$ for $i\in[n]$.
Suppose that $\mu_{i,l}>0$ for each $i$.  Then $f_{l+1}=f_{l}, g_{l+1}=\frac{f_l+g_l}{2}$.  If $\mu_{i,l}=0$ for some $i$ or the system \eqref{llinineqPin} with $\mathbf{b}=\0$ is not solvable set $f_{l+1}=\frac{f_l+g_l}{2}, g_{l+1}=g_{l}$.  Clearly $f_{l+1}\le \rho(A,B)\le g_{l+1}$.  Continue this bisection to $l=M$.  Then $\tilde\rho(A,B)=f_M$ and the condition \emph{(2)} holds.
\end{proof}

We now discuss briefly an approximation to $\hat\rho(A,B)$.  
First use use Algorithm \ref{alg}  to determine if $\hat\rho(A,B)<\infty$.  
Assume that $\hat\rho(A,B)<\infty$.  
Next assume that $\hat\rho(A,B)>0$, that is, the condition \emph{7} of Lemma \ref{rowB=0} holds.  Then the right hand side of \eqref{lowestrhoAB1} as a positive lower bound for $\hat\rho(A,B)$.
 Use Algorithm \ref{alg}  to find a nonempty subset $S\subseteq [n]$ such that $r(A,B,\1_S)<\infty$. Thus $r(A,B,\1_S)$ is an upper bound on $\hat\rho(A,B)$.
Next apply a simplified version of the bisection algorithm used in the proof of Theorem \ref{epsapproxCWn}, (without considering the minimum problem), to deduce:
\begin{proposition}\label{epsapproxwCWn}  Let $A,B\in\Q_+^{m\times n}\setminus\{0\}$.  Assume that $0<\hat\rho(A,B)<\infty$.  Then for any $\varepsilon\in (0,1)\cap \Q$ one can find $\bar\rho(A,B)\in \Q_{+}\setminus\{0\}$ in polynomial time in $\langle A\rangle+\langle B\rangle+\langle \varepsilon\rangle $, such that 
\begin{equation}\label{epsapproxwCWn1}
\bar\rho(A,B)\le \hat\rho(A,B)\le (1+\varepsilon)\bar\rho(A,B).
\end{equation}
\end{proposition}
\section{Minimal optimal solutions}\label{sec:minopt}
We first discuss weak optimal solutions which are easier to characterize.
 \begin{theorem}\label{CWpairthm} Let $m,n$ be positive integers.  Assume that $A=[a_{ij}],B=[b_{ij}]\in\R^{m\times n}_+$.  Suppose that $\hat\rho(A,B)<\infty$.
 \begin{enumerate} 
 \item Assume that $\y\in\R^n_+\setminus\{\0\}$ is weak optimal.  Then at least one coordinate of $(A-\hat\rho(A,B)B)\y$ is zero.
 \item Assume that there exists a weak optimal vector $\y\in\R^n_+$ with $\ell$ positive coordinates.  Let $A',B'\in\R^{m\times \ell}$ be the submatrices of $A$ and $B$ respectively induced by the positive coordinates of $\y$.  If $\ell\ge m$ then $\rank (A'-\hat\rho(A,B)B')<m$.
 \item
 A minimal weak optimal $\y$ has at most $m$ positive coordinates. 
 \item Let  $\y$ be a minimal weak optimal with $m$ positive coordinates.  Then $\y$ is a weak GPF-vector.  Furthermore $\rank (A'-\hat\rho(A,B)B')=m-1$.
\item Let $\y'$ be a minimal weak optimal with $\ell<m$ positive coordinates. Then there exists a minimal weak optimal $\y$, satisfying $\supp \y=\supp\y'$  with the following property:
Let $\cK=\{k\in[m], (A\y)_k=\hat\rho(A,B)(B\y)_k\}$.  Then $|\cK|\ge \ell$. 
 \end{enumerate}
 \end{theorem}
 \begin{proof}
 {\it (1)}  Let $D(t)=A-tB$ and $t_0=\hat\rho(A,B)$.  As $r(A,B,\y)=t_0$ it follows that $C(t_0)\y\le \0$.  Suppose to the contrary that $C(t_0)\y<\0$.  Then $t_0>0$.   Furthermore, there exists $0\le t_1<t_0$ such that $C(t_1)\y\le\0$.  Hence $r(A,B,\y)\le t_1<t_0$ contrary to our assumption.
 
 \noindent
 {\it (2)}  Assume that $\y$ is  weak optimal vector which has $\ell$ positive coordinates.  Let $A',B'\in\R^{m\times \ell}$ be defined as in the theorem.   Assume that $\0<\z\in\R^{\ell}$ is  the subvector of $\y$ induced by its  positive coordinates.  Let $C(t)=A'-tB'$ and $t_0=\hat\rho(A,B)$.  Then $C(t_0)\z=-\w, \w\in\R^m_+$.  Assume $\ell\ge m$ and $\rank C(t_0)=m$.  Then there exists a $m\times m$ submatrix of $C(t_0)$ which is nonsingular.  By permuting the columns of $C(t_0)$ we can assume the following.  Let $A'=[A_1\;A_2], B'=[B_1\;B_2]$ where $A_1,B_1\in\R^{m\times m}$ and $\det (A_1-t_0B_1)\ne 0$.  Denote $C_1(t)=A_1-tB_1, C_2(t)=A_2-tC_2$.  Assume that $\z\trans=(\uu\trans, \vv\trans), \0<\uu\in\R^m, \0<\vv\in \R^{\ell-m}$.  Thus 
 $C_1(t_0)\uu=-(C_2(t_0)\vv+\w)$.  Since $\det C_1(t_0)\ne 0$, there exists $\varepsilon>0$ such that $\det C_1(t)\ne 0$ for $|t-t_0|<\varepsilon$.
 Observe that $\uu=-C_1(t_0)^{-1}(C_2(t_0)\vv+\w)>\0$.  Let $\uu(t)= -C_1(t)^{-1}(C_2(t)\vv+\w)$ for $t\in(t_0-\varepsilon,t_0+\varepsilon)$.  Then $\uu(t)$ is continuous in the interval $(t_0-\varepsilon,t_0+\varepsilon)$.  Hence there exists
$\varepsilon_1\in (0,\varepsilon)$ such that $\uu(t)>\0$ for $|t-t_0|\le \varepsilon_1$. 
Set for $t_1=t_0-\varepsilon_1$ and $\z'=( \uu(t_1)\trans,\vv\trans)\trans$.
Thus $C(t_1)\z'=-\w$, which implies that $r(A',B',\z')\le t_1<\hat\rho(A',B')=\hat\rho(A,B)$.  This contradicts the definition of $\hat\rho(A',B')$.  Hence $\rank C(t_0)<m$. 
 
\noindent 
 {\it (3)}  Assume to the contrary that $\y$ is a minimal weak optimal solution with $\ell>m$ positive coordinates.  Let $A',B',C(t),\z,\w,t_0$ be defined as in part \emph{2} of the proof.
We showed that $\rank C(t_0)<m$.  Hence $\dim\ker C(t_0)\ge 2$.  Choose $\x\in\ker C(t_0)\setminus\{\0\}$ such that $\x$ has at least one negative coordinate and one positive coordinate.  So $\y$ is not proportional to $\x$.  Let $\z(s)=\z+s\x$ for $s\ge 0$.  Note that $C(t_0)\z(s)=-\w$.  Let $s_0>0$ be the biggest $s$ such that $\z(s)\ge \0$.  Then $\z(s_0)\gneq\0$, $\z(s_0)$ has at least one zero component and $C(t_0)\z(s_0)=-\w$.  Thus $r(A',B',\z(s_0))\le t_0=\hat\rho(A',B')$.   Hence $r(A',B',\z(s_0))=t_0$, which contradicts the minimality of $\y$.

\noindent
{\it (4)}  Assume that $\y$ a minimal weak optimal with $m$ positive coordinates.  Use the notations of parts \emph{2} and \emph{3} of the proof.  We claim that $C(t_0)\z=0$.  Suppose not.  By part \emph{2} $\rank C(t_0)<m$.  Let $\x\in\ker C(t_0)\setminus\{\0\}$.  So $\y$ is not proportional to $ \x$.  By considering $\pm \x$ we may assume that $\x$ has at least one negative coordinate.  Define $\z(s_0)$ as in part \emph{(3)} to deduce that $\y$ is not minimal.  So $C(t_0)\z=0$.  Assume to the contrary that $\rank C(t_0)<m-1$.  Choose $\x\in\ker C(t_0)$ to have positive and negative coordinates.  Then we conclude as above that $\z$ is not minimal.

\noindent
{\it (5)}  Let $\y'$ be a minimal weak optimal with $\ell<m$ positive coordinates.   Part \emph{1} of the theorem yields that $C(t_0)\y'$ has at least on zero coordinate.
Thus if $\ell=1$ part \emph{5} of the theorem is trivial.  

Assume that $\ell>1$.
Consider all minimal weak optimal $\tilde \y$ such that $\supp \tilde\y=\supp \y'$. 
Let $\cK(\tilde \y)=\{k\in[m], (A\tilde\y)_k=\hat\rho(A,B)(B\tilde\y)_k\}$.  Choose a minimal weak optimal $\y$  such that $|\cK(\y)|=p$ is maximal.  We claim that $p\ge \ell$.
Suppose not.  Let $\cK=\cK(\y)$.
Assume the notations of part \emph{2}.  Let $\tilde A=A'[\cK,[n]], \tilde B=B'[ \cK,[n]], \tilde C(t)=C(t)[\cK,[n]]$.
Hence $\tilde C(t_0))\z=\0$.  Suppose first that $\rank \tilde C(t_0)\le \ell-2$. Hence there exists $\uu\in \R^{\ell}$ satisfying $\tilde C(t_0)\uu=0$ such that $\uu$ has positive and negative coordinates. Thus $\z$ and $\uu$ are linearly independent.
Let $\vv\in \R^{m}$  be the extension of $\uu$ by adding zero coordinates.    
In particular, $\vv$ has a zero coordinate where $\y$ has zero coordinate.
Let $s\ge 0$ and consider $\z(s)=\z+s\uu$ and $\y(s)=\y+s\vv$.  Let $s_1>0$ be the smallest value such that $\z(s_1)\ge \0$ and $z(s_1)$ has at least one zero coordinate.  Since $\y$ was minimal we deduce that $\y(s_1)$ is not weak optimal.  That is, there exists $s_2\in (0,s_1)$ with the following property: There exists $j\in [m]\setminus\cK$ such that $(A\y(s_2))_j=t_0(B\y(s_2))_j$ and $(A\y(s_2))_k\le t_0(B\y(s_2))_k$ for $k\in[m]\setminus\{\cK\cup\{j\}\}$. Clearly $(A\y(s_2))_k= t_0(B\y(s_2))_k$ for $k\in \cK\cup\{j\}$.  So $\y(s_2)$ is optimal and $|\cK(\y(s_2))|>|\cK|=|\cK(\y)|$.  This contradicts the choice of $\y$.

It is left to consider the case where $|\cK|=\ell-1$ and $\rank \tilde C(t_0) = \ell-1$.
We proceed similarly as in the proof of \emph{(2)}.  Permute the columns of $C(t)$ such that $\tilde C(t)=[\tilde C_1(t)\, \tilde C_2(t)]\in \R^{(\ell -1)\times \ell}$ and $\tilde C_1(t_0)\in \R^{(\ell -1)\times (\ell-1)}$ is a nonsingular matrix.  Therefore $\tilde C_1(t)$ is nonsingular for $|t-t_0|<\varepsilon$ for some $\varepsilon>0$.
Assume that $\z\trans=(\uu\trans,\vv\trans)$.  Thus $\uu=-(\tilde C(t_0))^{-1}\tilde C_1(t_0)\vv$.  As $\uu>0$ it follows that $\uu(t)=-(\tilde C(t))^{-1}\tilde C_1(t)\vv >0$ for $|t-t_0|<\varepsilon'$ for some $\varepsilon'\in (0,\varepsilon)$.  Let $\w(t)\trans=(\uu(t)\trans,\vv\trans)$.  As $(A'\w(t_0))_j<(t_0-\varepsilon_1)(B'\w(t_0))_j$ for $j\in[m]\setminus\cK$ and some $\varepsilon_1>0$, there exists $t'\in(t_0-\varepsilon',t_0)$ such that $\tilde C(t')\w(t)=\0$ and $(A'\w(t'))_j<t'(B'\w(t'))_j$ for $j\in[m]\setminus\cK$.  That is, $\hat\rho(A',B')\le t'$.  This contradicts our assumption that $\hat\rho(A',B')=t_0$.
Therefore $|\cK|\ge \ell$.
 \end{proof}

Parts \emph{(4)-(5)} and their proof  yield:
 \begin{corollary}\label{subGPFeig} 
  Let $m,n$ be positive integers.  Assume that $A,B\in\R^{m\times n}_+$.  Suppose that $\hat\rho(A,B)<\infty$.  Then there exists a minimal weak optimal $\y\in\R^n_+$ with at most $m$ positive coordinates, and a minimal $\cI\subset [m]$, possibly an empty set, with the following properties:
\begin{equation}\label{subGPFeig1}
A(\cI,\emptyset)\y=\hat\rho(A,B)B(\cI,\emptyset)\y, \quad  m-|\cI|\ge |\mathrm{supp }\;\y|.
\end{equation}
Furthermore  if  $m-|\cI|=|\mathrm{supp}\, \y|$ then the only nonzero solution of $(A(\cI),\emptyset)-\hat\rho(A,B)B(\cI,\emptyset))\x=\0$ whose nonzero coordinates lie in $\mathrm{supp}\; \y$ are multiples of $\y$. 
\end{corollary}
We now show that some similar results apply to optimal vectors.
 \begin{theorem}\label{CWpairthm1} Let $m,n$ be positive integers.  Assume that $A=[a_{ij}],B=[b_{ij}]\in\R^{m\times n}_+$.  Suppose that $\rho(A,B)<\infty$.
 \begin{enumerate} 
 \item Assume that $\y\in\R^n_+\setminus\{\0\}$ is optimal.  Then at least one coordinate of $(A-\rho(A,B)B)\y$ is zero.
 \item Assume that there exists an optimal vector $\y\in\R^n_+$ with $\ell$ positive coordinates.  Let $A',B'\in\R^{m\times \ell}$ be the submatrices of $A$ and $B$ respectively induced by the positive coordinates of $\y$.  If $\ell\ge m$ then $\rank (A'-\rho(A,B)B')<m$.
 \item There exists an optimal $\y$ has at most $m$ positive coordinates. 
 \end{enumerate}
 \end{theorem}
 \begin{proof}  
 \emph{(1)}  Part \emph{2} of Lemma \ref{rhoABachiev} yields the existence of a  sequence $\y_l\in\Pi_n^o$ such that $\lim_{l\to\infty}\y_l=\y$, and 
 $\lim_{l\to\infty}r(A,B,\y_l)=r(A,B,\y)=\rho(A,B)$.
 Clearly, at least one coordinates of $(A-r(A,B,\y_l)B)\y_l$ is zero.  Hence there exists an infinite subsequence $\{l_p\},p\in\N$ such that a fixed coordinate of 
 $(A-r(A,B,\y_l)B)\y_{l_p}$ is zero.  Letting $p\to\infty$ we deduce the claim.
 
 \noindent
 \emph{(2)}  Let $t_0=\rho(A,B)$.  
 Assume first that $B\y>\0$.   
 Then part \emph{(3)} of Lemma \ref{scproprho} yields that $r(A,B,\x)\ge \rho(A,B)$ for each $\x\in\R^n_+, \sup\x=\sup\y$.
 Then we proceed as in the proof of part \emph{2} of Theorem \ref{CWpairthm}, using the notations and the results in this proof.  Rename the columns of $A$ and $B$ so that the first $\ell$ coordinates of $\y$ are positive.  Let $A'=A[[m],[\ell]], B'=B[[m],[\ell]], C(t)=A'-tB'$, and $\0<\z\in\R^\ell$ the projection of $\y$ on its first $\ell$ coordinates.
 As $B\y=B'\z>\0$ it follows that we can choose $t_1<t_0$ such that $B'\z'>\0$.  Hence $r(A',B',\z')\le t_1<t_0$ which contradicts the claim that $r(A',B',\z')\ge t_0$.  Hence $\rank C(t_0)<m$.
 
Assume that $\cI\ne\emptyset$ is the set of zero coordinates of $B\y$.  That is $\cI$ is the set of the zero rows of $B'$.    As $\rho(A,B)<\infty$ we deduce that the set of zero rows of $A'$ contains $\cI$.  Clearly, $\rank C(t_0)<m$.
 
 \noindent
 \emph{(3)}  We prove the claim by induction on $n$.  For $n=1$ the vector $\y=(1)$ is optimal, hence $m\ge \ell=1$.
 Assume that the claim holds for $n\le N$ and suppose that $n=N+1$.  Let $\y$ be a minimal optimal with vector with the minimal number of positive coordinates $\ell$. 
 Assume the conventions of parts \emph{(1)-(2)} of the proof.  Suppose to the contrary that $\ell>m$.   Then we proceed 
 to the proof of part \emph{(3)} of Theorem \ref{CWpairthm}.  Let $t_0$ and $\0<\z\in \R^{\ell}$ be as in part \emph{2}.  Let  $\cJ=\{\ell+1,\ldots,n\}$.  (If $\ell=n$ then $\cJ=\emptyset$.)
 Assume that $\x\in\ker C(t_0)\setminus\{\0\}$ such that $\x$ has at least one negative coordinate and one positive coordinate.  
 Set $\z(s)=\z+s\x$ and choose $s_0>0$ the biggest $s$ such that $\z(s_0)\ge\0$.  Then $\z'=\z(s_0)\gneq \0$ and $\z'$ has at least one zero coordinate.  Clearly, $r(A',B',\z')\le r(A',B',\z)=t_0$. 
 
 Suppose first that $r(A',B',\z')<t_0$.  Part \emph{(2)} of Lemma \ref{scproprho} yields that $B'\z'$ is not a positive vector. Hence $B'$ has at least one zero row.  Let $\cI$ be the set of zero rows of $B$.  As $r(A',B',\z')<t_0$ it follows that the set of zero rows of $A$ contains $\cI$.
 Hence $\hat\rho(A',B')\le r(A',B',\z')< t_0$.  
 
 Assume first that $|\cI|=m$.  Then $A'=B'=0$.  If $\cJ=\emptyset$ then $A=A',B=B'$, and a vector $(1,0,\ldots,0)\trans$ is a minimal optimal vector, contrary to our assumption.
 Hence $|\cJ|\ge 1$.  Clearly, $\rho(A,B)=\rho(A[[m],[\cJ]],B[[m],[\cJ]])$ and an optimal vector of $(A[[m],[\cJ]],B[[m],[\cJ]])$ can be extended trivially, by adding zero coordinates, to an optimal vector of $(A,B)$.  The induction hypothesis on the pair $(A[[m],[\cJ]],B[[m],[\cJ]])$ yields  the existence of an optimal $\w\in\R^{n-\ell}_+$ satisfying $|\sup \w|\le m$.  Extend trivially $\w$ to $\R_+^n$ to obtain an optimal vector $\y$ with $\sup\y\le m$, contrary to our assumption.
 
 Assume that $|\cI|=m-m'\in[m-1]$.  Rename the rows of $A$ and $B$ such that $\cI=\{m'+1,\ldots,m\}$.  Let $C=A[[m'],[\ell]],D=B[[m'],[\ell]]$.  Clearly $\hat\rho(C,D)=\hat\rho(A',B')<t_0$.
 Let us consider a weak minimal  vector $\w\in\R_+^\ell$  corresponding to $(C,D)$, with the minimal number of positive coordinates. 
 Theorem \ref{CWpairthm} yields that $p=|\sup \w|\le m'$.  Rename the first $\ell$ columns of $A$ and $B$ such that the first $p$ coordinates of $\w$ are positive.
 Let $\0<\uu\in\R^p$ be the projection of $\w$ on its first $p$-coordinates.
 Let $C'=C[[m'],[p]], D'=D[[m'],[p]]$
 Observe that 
 $\hat\rho(C,D)=r(C,D,\w)=r(C',D',\uu)=\hat\rho(C',D')<t_0$.  
 Let $\cI'$ be the set of zero rows of $D'$.  We claim that $m'\ge p+|\cI'|$.   
 Recall the set of zero rows in $C'$ contains $\cI'$.  Let $A_{11}=C'(\cI',\emptyset),
 B_{11}=D'(\cI',\emptyset)$. Clearly, $\hat\rho(C',D')=\hat\rho(A_{11},B_{11})$.  Furthermore,  $\x\in\R^p_+$ is a weak optimal vector of $(A_{11},B_{11})$ if and only if it is a weak optimal vector of $(C',D')$.  Note that each weak optimal vector of $(C',D')$ can be trivially extended to a weak optimal vector of $(C,D)$.  Since a weak optimal vector of $(C,D)$ has at least $p$ positive coordinates, it follows that each optimal vector of $(A_{11},B_{11})$ is positive.   Hence $\rho(A_{11},B_{11})=\hat\rho(A_{11},B_{11})<t_0$.
 
 Apply Theorem \ref{CWpairthm} to minimal optimal vectors of $(A_{11},B_{11})$ to deduce the inequality $m'-|\cI'|\ge p$.  Let $m''=m'-|\cI'|$ rearrange the first $m'$ rows of $A$ and $B$ such that $\cI=\{m''+1,\ldots,m'\}$. Let $\cK=\{p+1,\ldots,n\}$.  Observe
\begin{eqnarray*}
&&A=[A_1\;A_2],\;A_1=A[[m],[p]]=\left[\begin{array}{c}A_{11}\\0\end{array}\right],\; A_2=A[[m],[\cK]]=\left[\begin{array}{c}A_{12}\\A_{22}\end{array}\right],\\
&&B=[B_1\;B_2],\;B_1=B[[m],[p]]=\left[\begin{array}{c}B_{11}\\0\end{array}\right],\; B_2=B[[m],[\cK]]=\left[\begin{array}{c}B_{12}\\B_{22}\end{array}\right], \\
 &&A_{11}, B_{11}\in\R_+^{m''\times p},\quad A_{22}, B_{22}\in \R^{(m-m'')\times (n-p)}.
\end{eqnarray*}
As $\cI'$ was the set of zero rows of $D'$ we deduce that $B_{11}$ has no zero rows.  Hence $B_{11}\1_p>\0$.  Thus the conditions of part \emph{(3)} of Lemma \ref{scproprho}  hold.  Thus $\rho(A_{22},B_{22})=t_0$.  We now can apply the induction hypothesis on the optimal vector of the pair $(A_{22},B_{22})$.  That is, there exists an optimal $\vv\in\R_+^{n-p}$ satisfying $|\sup \uu|\le m-m''$.
We now use the arguments of the proof of part \emph{(2)} of Lemma \ref{rhoABachiev}
to deduce that $\y=(\uu\trans, s_1\vv\trans)\trans$ is an optimal for the pair $(A,B)$ for some $s_1>0$.  Note that $|\sup\y|= |\sup \uu|+|\sup \vv|\le p+m-m''\le m''+(m-m'')=m$, which contradicts our assumption.
  
It is left to discuss the case where $r(A',B',\z(s_0))=\rho(A,B)$.  Let $\hat\x\in\R^n$ be the trivial extensions of $\x$ to $\R^n$.  Then $A\hat\x=t_0\hat\x$.
Let $\y(s)=\y+s\hat\x\in\R^n$. Then $(A-t_0B)\y(s)=(A-t_0B)\y\le \0$.
Hence $r(A,B,\y(s))\le  t_0$ for $\y(s)\ge \0$.  This holds for $s\in[0,s_0]$.
Note that our assumption is that $r(A,B,\y(s_0))=t_0$.
We claim that $\y(s_0)$ is optimal.   Recall that there exists a sequence $\0<\y_k \in \R^n$ such that $\lim_{k\to\infty}\y_k=\y, \lim_{k\to\infty} r(A,B,\y_k)=r(A,B,\y)=t_0>0$.   Let $\y_l(s)=\y_l+s\hat\x$.  Then 
 \begin{eqnarray}
 (A-t_0B)\y_k(s)= (A-t_0B)\y_k\le (r(A,B,\y_k)-t_0)\y_k,\quad k\in\N.
\label{ylsineq}
 \end{eqnarray}
 Let $\y_k=\tilde\y_k+\vv_k$, where $\tilde\y_k$ is obtained by replacing the positive coordinates of $\y_k$ with zero coordinates in the places $\y$ has zero coordinates.
 Then $\vv_k=\y_k-\tilde\y_k\ge \0$.  The coordinates of $\vv_k$ are positive where the coordinates of $\tilde\y_k$ are zero, and the coordinates of $\vv_k$ are zero where the coordinates of $\tilde\y_k$ are positive.  Thus $\y_k(s)=\y(s)+(\tilde \y_k-\y)+\vv_k$. Fix $s\in (0,s_0)$. 
 We claim that there exists $K(s)>1$ and $N(s)>0$ such that for $k>N(s)$ we have the inequality $\y_k\le K(s)\y_k(s)$.  Indeed as $\y(s)$ has positive coordinates where $\y$ has positive coordinates it follows that there exists $K(s)>1$ such that $\y\le \frac{K(s)}{2}\y(s)$.  Clearly $\vv_k\le K(s)\vv_k$ for all $k\in\N$.  As $\lim_{k\to\infty} (\tilde\y_k-\y)=\0$ it follows that there exists $N(s)$ such that $(\tilde\y_k-\y)\le\frac{K(s)}{2}\y(s)$ for $k>N(s)$.  Hence $\y_k\le K(s)\y_k(s)$ and $\y_k(s)>\0$ for $k>N(s)$.  Therefore $r(A,B,\y_k(s))\ge t_0$ for $k>N(s)$.  Use  \eqref{ylsineq} to deduce that $r(A,B,\y_k(s))\le r(A,B,\y_k)$ for $k>N(s)$.

 Choose a sequence an increasing sequence $0<s_1<s_2<\cdots$ which converges to $s_0$.  Choose an increasing subsequence $l_j, j\in\N$ such that $\l_j>N(s_j)$,
 $|\y_{l_j}(s_j)-\y(s_j)|<\frac{1}{j}$. and $r(A,B,\y_j(s_j))\in [t_0,t_0+\frac{1}{j}]$.
 So $\lim_{j\to\infty}\y_{lj}(s_j)=\y(s_0)$ and $\y(s_0)$ is optimal.  This contradicts our assumption that $\y$ is optimal vector with the minimum number of positive coordinates.
 \end{proof}

We now give a simple example of two positive invertible stochastic matrices $A,B\in\R^{2\times 2}_+$ for which there is a unique optimal $\y\in\Pi_2$ with one positive coordinate.
 \begin{proposition} \label{2x2exam}  Let 
 \[A=\left[\begin{array}{cc}a&1-a\\b&1-b\end{array}\right], B=\left[\begin{array}{cc}1-a&a\\1-b&b\end{array}\right], \quad 0<b<a<\frac{1}{2}.\]
 Then $\rho(A,B)=\hat\rho(A,B)=\frac{a}{1-a}<1$ and $\z=(1,0)\trans$ is the unique optimal vector in $\Pi_2$, which is not a GPF-eigenvector.
 \end{proposition}
 \begin{proof} As $a,b\in(0,\frac{1}{2})$ it follows that 
 \[\min_{\x\in\Pi_2}\frac{(A\x)_1}{(B\x)_1}=\frac{(A\z)_1}{(B\z)_1}=\frac{a}{1-a}, \quad \min_{\x\in\Pi_2}\frac{(A\x)_2}{(B\x)_2}=\frac{(A\z)_2}{(B\z)_2}=\frac{b}{1-b},\]
 where $\z=(1,0)\trans\in\Pi_2$ is the unique vector that minimizes both ratios.  As $\frac{b}{1-b}<\frac{a}{1-a}$ we deduce that $\rho(A,B)=\frac{a}{1-a}$ and $\z$ is a unique optimal
 in $\Pi_2$. Clearly, $\rho(A,B)<1$.
 \end{proof}
 
 Note that $\rank (A-\rho(A,B)B)=2$, which does not contradict part \emph{2} of Theorem \ref{CWpairthm} as $\ell=1<m=2$.  Observe that $A\1_2=B\1_2$, that is $1$ is the eigenvalue of the generalized eigenvalue problem \eqref{geneigprobAB} with a corresponding positive eigenvector $\1_2$.  Note that the second eigenvalue of
 \eqref{geneigprobAB} is $\lambda=-1$ with a corresponding eigenvector $(1,-1)\trans$. 

 Recall that for $A\in\R^{n\times n}_+$ we have that $\rho(A\trans)=\rho(A)$.  For a pair of $A,B\in\R^{m\times n}_+$  such equality does not always hold.  For a pair of matrices given in Proposition \ref{2x2exam} we have $\rho(A\trans,B\trans)=\frac{1-a}{a}>\rho(A,B)$.  Indeed, observe that
 $\frac{(A\trans\x)_1}{(B\x)_1}<\frac{(A\trans\x)_2}{(B\x)_2}$ for each $\x\in\Pi_2$.  The minimum of the bigger ratio is achieved for $\z=(1,0)\trans$, which yields the equality
  $\rho(A\trans,B\trans)=\frac{1-a}{a}$.
 \section{A special $B$ appearing in a wireless network}\label{sec:specB}
 \begin{definition}\label{wcpair}  A pair $A,B\in\R_+^{m\times n}$ is called a WN-pair, (a wireless network pair), if
 $n\ge m$, $B$ has no zero row and each column of $B$ has exactly one positive entry. 
 \end{definition}
 A WN-pair was considered in \cite{ABHKLPP}.  It has the following interpretation in a wireless network \cite[Introduction]{ABHKLPP}.  Each row $i$  in $A=[a_{ij}]$ corresponds to the \emph{entity} $i$ (receiver), and each nonzero element in the row $i$ of $B=[b_{ij}]$ corresponds to an \emph{affector} (transmitter) of the entity $i$.  In the classical case, $m=n$ and $B$ is a diagonal matrix with positive diagonal.  That is,  each entity $i$ has one affector located at the entry $(i,i)$ of $B$.  In more general case the entity $i$ may have several affectors corresponding to the positive entries in the row $i$ of $B$.  The assumption that each column $B$ has one positive entry means that two different entities do not share a common affector.   In view of the wireless network interpretation of the entries of $A$ and $B$, it is assumed in \cite{ABHKLPP} that $a_{ij}b_{ij}=0$ for each pair $(i,j)$. In our treatment we drop this assumption. 
 
 Note that if $m=n$ then $B$ is called a monomial matrix.  So $B=PD$, where $P$ is an $m\times m$ permutation matrix and $D$ is an $m\times m$ diagonal matrix with positive diagonal entires.  Hence $B^{-1}=D^{-1}P\trans$. 
 
 The following theorem gives an explicit formula for $\hat\rho(A,B)$ of a WN-pair.
 \begin{theorem}\label{WGPFA-P}  Assume that $A,B\in\R^{m\times n}_+$ is a WN-pair.    Let $\cE(A,B)\subset \Pi_n$ be a finite set of vectors $\w$ that satisfy the following five conditions: 
 \begin{enumerate}
 \item The vector $\w\in\Pi_n$ has $\ell\le m$ nonzero coordinates.  
 \item Let $\cI$ be the set of zero rows of $B[[m],\supp \w]$.  Then $|\cI|=m-\ell$.
 (Hence $B[[m]\setminus\cI,\supp\w]$ is monomial.) 
 \item   $A[\cI,\supp \w]=0$.
 \item The matrix $C(\w)=B[[m]\setminus\cI,\supp \w]^{-1}A[[m]\setminus\cI,\supp \w]$ is irreducible.
 \item Let $\0<\z\in\R^{\ell}$ be the projection of $\w$ on $\supp \w$.  Then the vector  $\z$ is the PF-eigenvector of $C(\w)$.  
 \end{enumerate}
  The above conditions imply that $\rho(C(\w))=r(A,B,\w)$ and $A\w=\rho(C(\w))B\w$.
 Furthermore
 \begin{equation}\label{charhatrhoGPFA-P}
 \hat\rho(A,B)=\min\{r(A,B,\w),\;\w\in\cE(A,B)\}
 \end{equation}
 In particular, $\y$ is a minimal weak optimal if and only if $\y\in\cE(A,B)$ and $\y$ minimizes the right hand side of \eqref{charhatrhoGPFA-P}.  Furthermore, each minimal weak optimal is a weak GPF-eigenvector.
 \end{theorem}
 \begin{proof}  We first justify that  the assumption $|\cI|=m-l$ in part \emph{(2)} implies that $B_{11}=B[[m]\setminus\cI,\supp \w]$ is a monomial matrix.   Since each column of $B$ has exactly one nonzero entries it follows that $B_1=B[[m],\supp \w]$ has $\ell$ nonzero entries.  Since $|\cI|=m-\ell$ it follows that $B_{11}$ has $\ell$ nonzero rows.  That is, each row and column of $B_{11}$ has exactly one nonzero element. We next show that the conditions \emph{(1)-(5)} imply that $A\w=\rho(C)B\w$.  Let $A_1=A[[m],\supp\w]$ and $A_{11}=A[[m]\setminus\cI,\supp w]$.
 Since $\z$ is a PF-eigenvector of $C(\w)$ it follows that $A_{11}\z=\rho(C(\w))B_{11}\z$.  Let $A_{21}=A[\cI,\supp\w], B_{21}=B[\cI,\supp\w]$.  As $A_{21}=B_{21}=0$ we deduce that $A_1\w=\rho(C(\w))B_1\w$, which is equivalent to $A\w=\rho(C)B\w$.  Hence $r(A,B,\w)=\rho(C(\w))$.   
 
In particular, $\hat\rho(A,B)\le r(A,B,\w)$.  Denote by $\rho_1(A,B)$ the minimum in  \eqref{charhatrhoGPFA-P}.  Then $\hat\rho(A,B)\le \rho_1(A,B)$.  To show the equality \eqref{charhatrhoGPFA-P} it is enough to show that  a minimal weak optimal $\y$ is in $\cE(A,B)$.  We show this claim by induction on $m$.

For $m=1$ the equality \eqref{charhatrhoGPFA-P} trivially holds. Assume that each minimal weak optimal $\y$ is in $\cE(A,B)$ for $m\le M$.  Suppose that $m=M+1$.
Assume that $\y\in \Pi_n$ is a minimal weak optimal vector with the $\ell=|\supp \y|$.  Part \emph{3} of Theorem  \ref{CWpairthm} yields that $\ell\le m$.  
Set $\0<\z$ be the projection of $\y$ on its support. Let $\cJ=[n]\setminus \supp\y$,  $A'=A(\emptyset,\cJ)$ and $B'=B(\emptyset,\cJ)$.
 Then $\hat\rho(A,B)=\hat\rho(A',B')=r(A',B',\z)$. 
Denote by $\cI$ the set of the zero rows of $B'$.
As $\hat\rho(A,B)<\infty$ it follows that $\cI$ is a set of zero rows of $A'$.
Let $\tilde A=A(\cI,\cJ), \tilde B=B(\cI,\cJ)\in \R_+^{m'\times l}$.  Thus $\hat\rho(A,B)=\hat\rho(\tilde A,\tilde B)=r(\tilde A,\tilde B,\z)$.
As $B$ does not have a zero column it follows that $B'$ does not have a zero column.  As $\cI$ is the set of zero rows  of $B'$ it follows that $\tilde B$ does not have zero columns or zero rows.  As each column of $\tilde B$ has one positive element it follows that $\tilde B$ has exactly $\ell$ nonzero entries.  Hence $m'\le \ell$.

The equality $\hat\rho(A,B)=\hat\rho(\tilde A,\tilde B)=r(\tilde A,\tilde B,\z)$ yields that $\z$ is a weak optimal solution for $\hat\rho(\tilde A,\tilde B)=\hat\rho(A,B)$.  The assumption that $\y$ was minimal weak optimal yields that $\z$ a minimal weak optimal for $(\tilde A,\tilde B)$. 

Assume first that $m'<m$.  We apply the induction hypothesis to $\tilde A,\tilde B$ and a minimal $\z$ to deduce that $m'=\ell$.  So $\tilde B$ is a monomial matrix, $\hat\rho(\tilde A,\tilde B)=\rho(C)$ and $\tilde A\z=\rho(C)\tilde B\z$.   The induction hypothesis yields that $C$ is irreducible.  Hence $\y\in\cE(A,B)$ as we claimed.

It is left to discuss the case where $m'=m=\ell$.  In this case we use part \emph{4} of Theorem \ref{CWpairthm}.  So $\y$ is a weak GPF-eigenvector.  In particular $\tilde A\z=\hat\rho(A,B)\tilde B\z$.  Since $m'=m$ it follows that $\tilde B$ is a monomial matrix.  Observe next that $\rho(C)=\hat\rho(\tilde A,\tilde B)=\hat\rho(C,I)$.  Furthermore, $\z>\0$ is a minimal weak optimal vector of $(C,I)$.
Apply now part \emph{(4)} of Theorem \ref{charrhoAI}.  A minimal weak optimal vector of $(C,I)$ is supported on $V_i\subseteq [m]$ which corresponds to a source in the reduced graph $\vec{G}_{red}$.   Furthermore $A[V_i]$ is irreducible.  Since $\z>\0$ it follows that $V_i=[m]$ and $C$ is irreducible.  Hence $\y\in\cE(A,B)$.
 \end{proof}
 The following notion of $S$-irreducibility was introduced in \cite{ABHKLPP}:
 \begin{definition} A WN-pair $A,B\in \R_+^{m\times n}$ is called $S$-irreducible if the following condition holds:  For each subset $\cK\subseteq [n]$ of cardinality $m$, such that $B[[m],\cK]$ is a monomial matrix,  the matrix $B[[m],\cK]^{-1}A[[m],\cK]$ is irreducible.
 \end{definition}

 Note that if $n=m$ then  $S$-irreducibility is equivalent to the irreducibility of $B^{-1}A$.
 The following proposition gives a sufficient condition for an $S$-irreducible pair:
 \begin{proposition}\label{scGPFA-P}  Let $A=[a_{ij}],B=[b_{ij}]\in\R_+^{m\times n}$ be a WN-pair.   Assume that $a_{ij}>0$ if $b_{ij}=0$.
 Then the pair $A,B$ is irreducible.
 \end{proposition}
 \begin{proof}  Let $\cK\subseteq [n]$ of cardinality $m$, such that $B_1=B[[m],\cK]$ is a monomial matrix.  Let $A_1=A[[m],\cK]$.  So $A_1$ has positive elements where the elements of $B_1$ are zero.  Therefore all off-diagonal entries of $C=B_1^{-1}A_1$ are positive, and $C$ is irreducible.  
 \end{proof}
 The following theorem gives an explicit formula for $\rho(A,B)$ of WN-pair.
 \begin{theorem}\label{rhoABWC}  Let $A,B\in\R_+^{m\times n}$ be a WN-pair.
 Denote by $\cM(B)$ the subset of all $\cK\subseteq[n]$ of cardinality $m$ such that the matrix $B[[m],\cK]$ is monomial.  Then
 \begin{equation}\label{rhoABWCfor}
 \rho(A,B)=\min\{\rho(B[[m],\cK]^{-1}A([m],\cK]),\;\cK\in\cM(B)\}
 \end{equation}
 For each $\cK\in\cM(B)$ such that $\rho(A,B)=\rho(B[[m],\cK]^{-1}A([m],\cK])$ there exists an optimal $\y\in\Pi_n$ with the following property: The support of $\y$ is contained in $\cK$ and $\y$ is a GPF-vector.
 Assume that WN-pair is $S$-irreducible. Then each such $\y$ is minimal optimal.
 \end{theorem}
 \begin{proof}
 Assume first that the pair $A,B$ is $S$-irreducible.  Let $\cE(A,B)$ be defined as in the Theorem \ref{WGPFA-P}.  We claim that $|\supp \w|=m$ for each $\w\in\cE(A,B)$.  
 Assume to the contrary that $\ell=|\supp \w|<m$.  After relabeling the elements of the set $[n]$ we can assume that $\supp\w=[\ell]$.  Let $\cI$ be zero set of $B_1=B[[m],[\ell]]$.  Recall that $B_{11}=B[[m]\setminus\cI,[\ell]]$ is monomial, and $A_{21}=A[\cI,[\ell]]=0$. 
 Relabel the elements of $[m]$ such that $\cI=\{\ell+1,\ldots,m\}$.  Since  $B$ does not have zero rows there is a subset 
 $\cJ$ of $[n]$ of cardinality $m-\ell$ such that $B[\cI,\cJ]$ is a monomial matrix.  Clearly,
 $[\ell]\cap \cJ=\emptyset$.  Let $\cK=[\ell]\cup \cJ$.  Then $B_1=B[[m], \cK]$ is a monomial matrix, which is direct sum of $B_{11}$ and   $B[\cI,\cJ]$.  Let $A_1=A[[m], \cK]$.  Since $A,B$ is $S$-ireducible it follows that  $C=B_1^{-1}A_1$ is irreducible.  This contradicts the fact that $C[\cI,[\ell]]=0$.
 
 Thus for each $\w\in\cE(A,B)$ we have that $|\supp \w|=m$.  Let $\z\in\Pi_m^o$ be the projection of $\w$ on $\supp\w$.  Then $C(\w)=B[[m],\supp\w]^{-1}A[[m],\supp\w]$ is irreducible, and $\z$ is the PF-vector of $C(\w)$.  Hence $B\w=B[[m],\supp\w]\z>\0$.  Theorem \ref{WGPFA-P} yields that $\rho(C(\w))=r(A,B,\w)$. 
 
 Vice versa, each $\cK\in \cM(B)$ induces $\w\in\cE(A,B)$ as follows.  Let $\z\in\Pi_m^o$ the PF-eigenvector of  $C=B[[m],\cK]^{-1}A([m],\cK])$.  Then $\w\in\Pi_n$ is obtained from $\z$ by adding zero coordinates.  As $C$ is irreducible we deduce that $\w\in\cE(A,B)$.
  
Recall that $\hat\rho(A,B)$ is given by \eqref{charhatrhoGPFA-P}.  Assume that $\y$ is a minimal weak optimal.  So $\y\in\cE(A,B)$ and $B\y>\0$.  Lemma \ref{Bpos} yields that $\y$ is minimal optimal.  Hence $\rho(A,B)=\hat\rho(A,B)$. Characterization \eqref{rhoABWCfor} follows from \eqref{charhatrhoGPFA-P}.   Furthermore $\y$ is GPF-vector.   This proves the theorem in the case where $A,B$ is $S$-irreducible.
 
 Assume now that $A,B$ is not $S$-irreducible.  Let $J\in\R^{m\times n}$ be a matrix whose all entries are $1$.  For $l\in\N$ denote let $A_l=A+\frac{1}{l}J$.  So $A_l>0$, and Proposition \ref{scGPFA-P} implies that the pair$A_l,B$ is $S$-irreducible.  Fix $\cK\in \cM(B)$.
 Let $C=B[[m],\cK]^{-1}A([m],\cK])$ and $C_l=B[[m],\cK]^{-1}A_l([m],\cK])$.
 Observe that $A_{l+1}\le A_l$ for $l\in\N$.  Hence $C_{l+1}\le C_l$ for $l\in\N$, and
  $\lim_{l\to\infty} C_l=C$.  Therefore $ \rho(C_l),l\in\N$ is a decreasing sequence which converges to $\rho(C)$.
 Apply characterization \eqref{rhoABWCfor} to $\rho(A_l,B)$.  Let $\rho_1(A,B)$ be the right hand side of \eqref{rhoABWCfor}.  It now follows that $\lim_{l\to\infty}\rho(A_l,B)=\rho_1(A,B)$.  Part \emph{5} of Lemma \ref{Bpos} yields that $\rho(A,B)=\rho_1(A,B)$.  Hence the characterization \eqref{rhoABWCfor}  holds.
 
 Assume that $\cK\in\cM(B)$ and $\rho(A,B)=\rho(B[[m],\cK]^{-1}A([m],\cK])$.
 For each $l\in\N$ let $\w_l\in\Pi_n$ be be the vector induced by the PF-eigenvector $\0<\z_l\in\Pi_m^o$ of  $B[[m],\cK]^{-1}A_l([m],\cK])$.  Let $r_l=\rho(B[[m],\cK]^{-1}A_l([m],\cK])$.  Then $A_l\w_l=r_lB\w_l$.  Pick a convergent subsequence
 $\w_{l_k}\to \y, k\to\infty$.  Then $A\y=\rho(A,B)B\y$. 
 
 We claim that $\y$ is optimal.  Choose $\x_l\in\Pi_n^o$ such that $r(A_l,B,\x_l)\le (\rho(A_l,B)+\frac{1}{l})\x_l$.  Clearly,  $\rho(A,B)\le\rho(A_l,B)\le r_l$.  Set $\vv_l=(1-\frac{1}{l})\y+\frac{1}{l}\x_l>\0$. Then $\lim_{k\to\infty} \vv_{l_k}=\y$, and 
 $\lim_{k\to\infty} r(A_l,B,\x_l)=\rho(A,B)$.  Hence $\y$ is optimal.
 \end{proof}
 
 To summarize, if a WN-pair is $S$-irreducible then each minimal optimal $\y$, which is a GPF-vector,  and corresponds to an optimal choice of one transmitter for each receiver.
 If a WN-pair is not $S$-irreducible there exist an optimal $\y$, which is a GPF-vector, and corresponds to an optimal choice of one transmitter for each receiver.  However, for some receivers all their transmitters may shut off.  
This can happen in the classical case where $m=n$ and $B=I_m$.  For example:
\[A=\left[\begin{array}{cccc}0&1&1&1\\1&0&1&1\\0&0&0&1\\0&0&1&0\end{array}\right].\]
Then the only optimal $\y\in\Pi_4$ is $\y=\frac{1}{2}(1,1,0,0)\trans$.
 \section{A pair of CP-operators}\label{sec:CPop}
 Recall that $\C^n$ is equipped with the standard inner product $\langle\x,\y\rangle=\y^*\x$.
Given a finite dimensional inner product space over $\C$, with a product $\langle\cdot,\cdot\rangle$, we denote by $\rS(\V)\supset \rS_{+}(\V)\supset \rS_{+,1}(\V)$ the real space of self adjoint operators $A:\V\to\V$, the cone of positive semidefinite operators and the convex set of all positive semidefinite operators with trace $1$.  By fixing an orthonormal basis $\e_1,\ldots,\e_n$ in $\V$ we identify
$\rS(\V), \rS_{+}(\V), \rS_{+,1,}(\V)$ with $\rH_n,\rH_{+,n},\rH_{+,1,n}$ respectively.

On $\C^{n\times n}$ we have the standard inner product $\langle U,V\rangle=\tr V^*U$.  For $X\in\rH_n$ we denote by $\lambda_1(A)\ge \cdots\ge \lambda_n(X)$ the $n$-real eigenvalues of $X$ counted with their multiplicities.
 Recall that for $X,Y\in\rH_{n}$ we say that $Y\succeq X$ if $Y-X\in\rH_{+,n}$, i.e., $Y-X$ is positive semidefinite.  Denote by $\rH_{++,n}$ the interior of the cone $\rH_{+,n}$, i.e., the open set of positive definite $n\times n$ Hermitian matrices.  Then $Y\succ X$ if $Y-X\in\rH_{++,n}$.  Let $\rH_{++,1,n}=\rH_{+,1,n}\cap \rH_{++,n}$.
 Denote by $CP(n,m)$ the cone of completely positive operators from $\rH_n$ to $\rH_m$, given by \eqref{defccpop}.
 In the rest of the paper  we assume that $\cA,\cB\in CP(n,m)$.  Then we can define $\rho(\cA,\cB)$  and $\hat\rho(\cA,\cB)$ as in \eqref{defrABxK},\eqref{defrhoABK} and \eqref{defhrhoABK} with respect to the cones $\mathbf{K}_1=\rH_{+,n},\mathbf{K}_2=\rH_{+,m}$.   We call $Y\in\rH_{+,n}\setminus\{0\}$  weak optimal if  $\hat\rho(\cA,\cB)=r(\cA,\cB,Y)$. 
We call $Y\in\rH_{+,n}\setminus\{0\}$ optimal if the following conditions hold:  First,    $\rho(\cA,\cB)=r(\cA,\cB,Y)$.  Second, there exists a sequence $X_l\in \rH_{++,n}$ such that $\lim_{l\to\infty} X_l=Y$ and $\lim_{l\to\infty} r(\cA,\cB,X_l)=r(\cA,\cB,Y)$.   We say that $Y$ is a generalized Perron-Frobenius vector or weak generalized Perron-Frobenius vector if
 \begin{equation}\label{GPFcpop}
 \cA(Y)=\rho(A,B)\cB(Y), \textrm{ or } \cA(Y)=\hat \rho(A,B)\cB(Y),\quad Y\in\rH_{+,n}\setminus\{0\},
 \end{equation}
 respectively.  
 
 Given a pair $\cC,\cD\in CP(n,m)$ we say that $\cC\succeq \cD$ or 
 $\cC\succ \cD$ if for each $X\in\rH_{+,1,n}$ we have that $\cC(X)\succeq \cD(X)$ or 
 $\cC(X)\succ \cD(X)$ respectively.  If $\cD=0$ the $\cC\succ 0$ is called a positive CP-operator.  An example of positive CP-operator $\cI(n,m)$ is the operator  $\cI(n,m)(Z)=( \tr Z)I_m$ for any $Z\in\C^{n\times n}$.  (We will justify briefly why $\cI(n,m)$ is completely positive in the next section.)
 
 In this paper we will concentrate on $\hat\rho(\cA,\cB)$ since this quantity is much easier to deal with.  When the proofs of our results for CP-pair very similar for the matrix pair $A,B\in\R_+^{m\times n}$ we will omit the proofs.

 For $\cA,\cB\in CP(n,m)$ and $X\in\rH_{n,+,1}$ we give a formula to compute $r(\cA,\cB,X)$.  To do that we need to recall the classical definition of the Rayleigh quotient for $A,B\in\rH_{+,m}$ \cite[\S4.4]{Fri15}:
 \begin{lemma}\label{releigAB}  Let $A,B\in\rH_{+,m}$.  Define
\begin{equation}\label{releigAB1}
\lambda(A,B)=\sup_{\x\in\C^m\setminus \{\0\}}\frac{\x^*A \x}{\x^*B\x}.
\end{equation}
Then the above supremum is achieved for some $\y\in\C^m\setminus \{\0\}$:
\begin{enumerate}
\item $\lambda(A,B)=\infty$ if and only if $\ker B$ is not a subset of $\ker A$.
Then $\y\in\ker B\setminus \ker A$.
\item If $A=B=0$ then $\lambda(A,B)=0$ and
$\y$ is any nonzero vector in $\C^m$.  
\item Assume that $\ker B\subseteq \ker A$ and $\dim\ker B<m$.  Let $\V\subseteq \C^m$ be the orthogonal complement of $\ker B$. Then $\V$ is an invariant subspace of $A$ and $B$.  Denote
by $A_1,B_1$ the restricitions of $A,B$ to $\V$.  Then $B_1^{-1}A_1$ is a diagonalizable operator in $\V$, with nonnegative eigenvalues.   Furthermore
\begin{equation}\label{releigAB2}
\lambda(A,B)=\lambda(A_1,B_1)=\rho(B_1^{-1}A_1).
\end{equation}
Morevover, a maximizing $\y$ of the quotient \eqref{releigAB1} can be chosen to be an eigenvector of $B_1^{-1}A_1$ corresponding to $\rho(B_1^{-1}A_1)$.
\end{enumerate}
In particular
\begin{equation}\label{infcharlambAB}
\lambda(A,B)=\inf\{t\ge 0,\; tB\succeq A\}.
\end{equation}
\end{lemma}
\begin{proof} Parts \emph{(1)} and \emph{(2)} are straightforward.  We now prove \emph{(3)}.  Suppose first that $\ker B=\{\0\}$.  So $B\in\rH_{++,m}$.  Let $C=\sqrt{B}$ be the unique root of $B$ such that $C\in\rH_{++,m}$.  Set $\x=C^{-1}\z$.  Then 
$\frac{\x^*A \x}{\x^*B\x}=\frac{\z^*C^{-1}A C^{-1}\z}{\z^*\z}$.  Thus the supremum  \eqref{releigAB1} is the maximum characterization of the maximum eigenvalue of $C^{-1}AC^{-1}\in\rH_{+,m}$.  So \[\lambda(A,B)=\lambda(C^{-1}AC^{-1},I_m)=\rho(C^{-1}AC^{-1})
=\rho(C^{-2}A)=\rho(B^{-1}A).\]
As $C^{-1}AC^{-1}\in\rH_{+,m}$ has nonnegative eigenvalues it follows that $B^{-1}A$ is diagonizable with nonnegative eigenvalues.

Suppose now that $1\le\dim\ker B<m$.  Let $\V=\ker B^{\perp}$.  Then $A\V\subseteq \V=B\C^n$.  Suppose that $\x\in\ker B\setminus\{\0\}$.  Then $\frac{\x^*A \x}{\x^*B\x}=\frac{0}{0}=0$.  Assume that $\x\in \C^m\setminus\ker B$.  Then $\x=\uu+\vv$, where $\uu\in\V\setminus\{\0\}$ and $\vv\in\ker B$.  Clearly$\frac{\x^*A \x}{\x^*B\x}=\frac{\uu^*A\uu}{\uu^*B\uu}=\frac{\uu^*A_1\uu}{\uu^*B_1\uu}$.
Apply the previous arguments to $A_1,B_1\in\rS_+(\V)$ to deduce that $\lambda(A,B)=\lambda(A_1,B_1)=\rho(B_1^{-1}A_1)$.

The characterization \eqref{infcharlambAB} follows straightforward from \eqref{releigAB1}.
\end{proof}
 \begin{corollary}\label{rcAcBXfor} Let $\cA,\cB\in CP(n,m)$ and $X\in\rH_{+,n}\setminus \{0\}$.  Then $r(\cA,\cB,X)=\lambda(\cA(X),\cB(X))$.
 \end{corollary}
 \begin{lemma}\label{rhoABachievqc}  
 Let $\cA,\cB\in CP(n,m)$.   Then
 \begin{enumerate}
 \item $\hat\rho(\cA,\cB)\le \rho(\cA,\cB)$.
 \item Assume that $\cA_1,\cB_1\in CP(n,m)$ and $\cA_1\preceq \cA, \cB\preceq \cB_1$. Then $\rho(\cA_1,\cB_1)\le \rho(\cA,\cB)$ and $\hat\rho(\cA_1,\cB_1)\le \hat \rho(\cA,\cB)$.
 \item There exists a weak optimal $Y\in\rH_{n,+,1}$.
 \item Assume that there exists a weak optimal $Y$ such that either $\cB(Y)\succ 0$ or $\cA(Y)\succ 0$.  If $\rho(\cA,\cB)<\infty$ then $\rho(\cA,\cB)=\hat\rho(\cA,\cB)$.
 \item Suppose that $\rho(\cA,\cB)<\infty$, and either $\cA\succ 0$ or $\cB\succ 0$.  Then $\rho(\cA,\cB)=\hat\rho(A,B)$.
 \item Assume that $0\prec \cD_l\in CP(n,m)$ for $l\in \N$ and $\lim_{l\to\infty} \cD_l=0$.  Then
$\lim_{l\to\infty}\rho(\cA,\cB+\cD_l)=\hat\rho(\cA,\cB)$.
\item Assume that $0\prec \cD_l\in CP(n,m)$ for $l\in \N$ and $\lim_{l\to\infty} \cD_l=0$.
If $\cB(I_n)\succ 0$ then 
$\lim_{l\to\infty}\rho(\cA+\cD_l,\cB)=\rho(\cA,\cB)$.
 \end{enumerate}
 \end{lemma}
 \begin{proof}
 \emph{(1)}  Trivial.
 
 \noindent
 \emph{(2)}  Straightforward from the definitions.
 
 \noindent
 \emph{(3)}.  As in the proof of part \emph{(1)} of Lemma \ref{rhoABachiev}.
 
 \noindent
 \emph{(4)}  As in the proof of part \emph{(2)} of Lemma \ref{Bpos}.
 
 \noindent
 \emph{(5)}  As in the proof of part \emph{(3)} of Lemma \ref{Bpos}.
 
 \noindent
 \emph{(6)}  We use similar arguments to the proof of part \emph{(4)} of Lemma \ref{Bpos} with the following modifications.  Let $\cB_l=\cC_l+\cD_l$, where $\cC_l=\cB$.   Then the arguments of the part \emph{(4)} of Lemma \ref{Bpos} apply.
 
 \noindent
 \emph{(7)}  We use similar arguments to the proof of part \emph{(5)} of Lemma \ref{Bpos} with the following modifications.  Let $\cA_l=\cC_l+\cD_l$, where $\cC_l=\cA$.   Then the arguments of the part \emph{(5)} of Lemma \ref{Bpos} apply.
 \end{proof}
 Let $\cC:\rH_{n}\to H_m$ be a linear operator.  Then there exists a dual operator $\cC^{\vee}:\rH_m\to \rH_n$ which is defined as follows.  Recall that on $\rH_n$ one has the inner product $\langle X,Z\rangle=\tr XZ$, where $\tr W$ is the trace of the matrix $W\in\C^{n\times n}$.  Then $\cC^\vee: \rH_m\to\rH_n$ is defined uniquely by the property $\langle\cC(X),Z\rangle=\langle X,\cC^\vee(Z)\rangle$ for all $X\in\rH_n$ and $Z\in\rH_m$.   Assume that $\cC$ is CP-operator given by \eqref{defccpop}.  Then $\cC$ is called \emph{unital} if $\cC(I_n)=I_m$.   Recall that  $\cC^\vee$ is also completely positive and given by
 $C^\vee(Y)=\sum_{j=1}^k T_j^*YT_j$.  Thus $\cC$ is a quantum channel if and only if $\cC^\vee$ is unital.  
 
The following Lemma is an analog of Lemma \ref{lowestrhoAB}:
\begin{lemma}\label{lowbdcp}Let $\cA,\cB\in CP(n,m)$.  Then 
\begin{enumerate}
\item $\hat\rho(\cA,\cB)=0$ if and only if $\cA^{\vee}(I_m)$ is not positive definite.
\item Assume that $\cA^{\vee}(I_m)$ is positive definite.  Then 
\begin{equation}\label{lowbdrhABcp}
\hat\rho(\cA,\cB)\ge \rho(\cA^{\vee}(I_m)^{-1}\cB^{\vee}(I_m))^{-1}.
\end{equation}
In particular, if $\cA$ and $\cB$ are quantum channels then $\rho(\cA,\cB)\ge 1$. 
\end{enumerate}
 \end{lemma}
 \begin{proof}  \emph{(1)}  Clearly, $\hat\rho(\cA,\cB)=0$ if and only if there exists $X\in\rH_{+,1,n}$ such that $\cA(X)=0$.  Recall that $Y\in\rH_{m,+}$ is zero if and only if $\langle Y,I_m\rangle=0$.  So $\cA(X)=0$ if and only if $\langle \cA(X),I_m\rangle=\langle X, \cA^{\vee}(I_m)\rangle=0$.  Hence $\cA^{\vee}(I_m)\not\succ 0$.   
 
 Vice versa, assume that $\cA^{\vee}(I_m)\not\succ 0$.   Therefore there exists a vector $\x\in\C^n$, $\x^*\x=1$, such that $\cA^{\vee}(I_m)\x=\0$.  In particular, 
 \[0=\x^*\cA^{\vee}(I_m)\x=\langle \x\x^*, \cA^{\vee}(I_m)\rangle =\langle \cA(\x\x^*),I_m\rangle \Rightarrow \cA(\x\x^*)=0.\]
 \noindent
 \emph{(2)}  Assume that $\cA^{\vee}(I_m)\succ 0$.  Let $X\in\rH_{n,+,1}$.  Assume that $\rank X=r\in[n]$.  Then spectral decomposiiton of $X$ is $X=\sum_{i=1}^r \lambda_i\x_i\x_i^*$, where each $\lambda_i>0$ and $\x_j^*\x_i=\delta_{ij}$ for $i,j\in[r]$.  As $\cA(X)\le r(\cA,\cB,X) \cB(X)$ it follows that
 \[\langle X,\cA^{\vee}(I_m)\rangle=\langle \cA(X),I_m\rangle\le r(\cA,\cB,X)\langle \cB(X),I_m\rangle=r(\cA,\cB,X)\langle  X,\cB^{\vee}(I_m)\rangle\]
 Hence
 \begin{eqnarray*}&&r(\cA,\cB,X)^{-1}\le \frac{\langle  X,\cB^{\vee}(I_m)\rangle}{\langle  X,\cA^{\vee}(I_m)\rangle}=\frac{\sum_{i=1}^r \lambda_i \x_i^*\cB^{\vee}(I_m)\x_i}{\sum_{i=1}^r \lambda_i \x_i^*\cA^{\vee}(I_m)\x_i}\\
&& \le\max_{i\in[r]}\frac{\x_i^*\cB^{\vee}(I_m)\x_i}{\x_i^*\cA^{\vee}(I_m)\x_i}\le \lambda(\cB^{\vee}(I_m),\cA^{\vee}(I_m))=\rho(\cA^{\vee}(I_m)^{-1}\cB^{\vee}(I_m)).
 \end{eqnarray*}
  This establishes  \eqref{lowbdrhABcp}.
 
 Assume that $\cA$ and $\cB$ are quantum channels.  Then $\cA^{\vee}(I_m)=\cB^{\vee}(I_m)=I_n$.  Hence $\rho(\cA^{\vee}(I_m)^{-1}\cB^{\vee}(I_m))=1$ and then $\hat\rho(\cA,\cB)\ge 1$.
 \end{proof}
 \section{Polynomial approximation of $\rho(\cA,\cB)$ for $\delta$-positive $\cB$}\label{sec:polaprhocAB}
 In this section we assume that $\cA,\cB \in CP(n,m)$.  Suppose furthermore that $0<\rho(\cA,\cB)<\infty$.
We now want to apply the bisection algorithm as in the proof of Theorem \ref{epsapproxCWn}.   To this end, for a given $t>0$ we need to decide if the intersection $(t\cB-\cA)(\rH_{+,1,n})\cap \rH_{+,m}$ is empty or not. 
  
Let $X\in\rH_{n}$ and $Z\in \rH_{m}$.  The square of the distance between $t\cB(X)-\cA(X)$ and $Z$ is given by the following quadratic convex function: 
\begin{equation}\label{defftfunc}
f_t(X,Z)=\langle t\cB(X)-\cA(X)-Z,t\cB(X)-\cA(X)-Z\rangle, \quad X\in\rH_n, Z\in\rH_m.
\end{equation}
We assume that $(X,Z)$ are in the cone $\rH_{+,n}\times\rH_{+,m}$ subject to the linear constrain $\tr X=1$.  Note that $f_t(X,Z)\ge 0$.    Finding the distance between the two convex sets $(t\cB-\cA)(\rH_{+,1,n})$ and $\rH_{+,m}$ is equivalent to the  minimization problem 
\begin{equation}\label{minftprob}
\mu_0(t)=\min\{f_t(X,Z), \;X\in\rH_{+,1,n}, Z\in\rH_{+,m}\}.
\end{equation}
This minimization problem can be dealt with by the standard interior point methods \cite{BV04}.  Fix $\tau>0$.   Assume that we found an approximation $\mu(t)\in\Q_{++}$ of $\mu_0(t)$ by an interior method within precision $\tau$ in polynomial time in the data.  If $\mu(t)\ge 2\tau$ then $\dist((t\cB-\cA)(\rH_{+,1,n}),\rH_{+,m})\ge \tau$.  Hence $\rho(\cA, \cB)> t$.  Suppose that $\mu(t)< 2\tau$.  How can we estimate from above $\rho(A,B)$?  
Recall that $\mu(t)=f_t(X(t),Z(t))$.  So 
\begin{eqnarray}\label{rABcpapr}
&&t\cB(X(t))=\cA(X(t))+Z(t)+W, \\ 
&&X(t)\in\rH_{++,1,n}, Z(t)\in \rH_{++,m}, W\in \rH_m, \|W\|^2=\tr W^2=\mu(t)<2\tau.\notag
\end{eqnarray}
To find an upper bound for $\rho(\cA,\cB)$ from \eqref{rABcpapr} we need to assume a positivity condition on $\cB$.    
\begin{definition}\label{deltaposL}  Assume that $m,n$ are two positive integers and $\delta\ge 0$.  Then
\begin{enumerate}
\item  Denote by $\cI(n,m):\rH_n\to \rH_m$ the linear transformation $\cI(n,m)(X)=(\tr X)I_m$ for $X\in\rH_{n}$.
\item
A real linear transformation $\cL:\rH_n\to \rH_m$ is called $\delta$-positive if   $\cL-\delta \cI(n,m)$ is completely positive.
\end{enumerate}
\end{definition}
Let $\cL:\rH_n\to \rH_m$ be a real linear transformation.  Since any $F\in\C^{n\times n}$ is of the form $F=X+\bi Y$, where $X,Y\in\rH_n$ and $(\bi )^2=-1$, it follows that $\C^{n\times n}$ is the complexification of $\rH_n$.  Hence $\cL$ extends to linear operators $\hat\cL:\C^{n\times n} \to \C^{m\times m}$ over $\C$ by letting $\hat\cL(\bi X)=\bi\cL(X)$ for $X\in\rH_n$.  We will identify $\hat\cL$ with $\cL$ and no confusion will arise.  Note that for $U\in\C^{n\times n}$ we have that $\cL(U^*)=\cL(U)^*$.  Associate with $\cL$ the following block Hermitian matrix of dimension $mn$:
\begin{equation}\label{defZL}
Z(\cL)=\left[\begin{array}{cccc}\cL(\e_1\e_1^*)&\cL(\e_1\e_2^*)&\cdots&\cL(\e_1\e_n^*)\\
\cL(\e_2\e_1^*)&\cL(\e_1\e_2^*)&\cdots&\cL(\e_2\e_n^*)\\ 
\vdots&\vdots&\vdots&\vdots\\
\cL(\e_n\e_1^*)&\cL(\e_n\e_2^*)&\cdots&\cL(\e_n\e_n^*)
\end{array}\right], \e_i=(\delta_{1i},\ldots,\delta_{ni})\trans, i\in[n].
\end{equation}
Denote by $\lambda_{mn}(Z(\cL))$ the smallest eigenvalue of $Z(\cL)$.
The following lemma follows from Choi's characterization of completely positive operators \cite{Cho75, FL16}.
\begin{lemma}\label{delposchar}  Let $\delta\ge 0$.  A real linear transformation $\cL:\rH_{n}\to \rH_m$  is $\delta$-positive if and only if 
$\lambda_{mn}(Z(\cL))\ge \delta$.
\end{lemma}
\begin{proof}Recall Choi's theorem \cite{Cho75} that $\cL-\delta\cI(n,m)$ is completely positive if and only if $\lambda_{mn}(Z(\cL-\delta\cI(n,m)))\ge 0$.
Clearly, for $U\in \C^{n\times n}$ we have that $\cI(n,m)(U)=(\tr U)I_m$.   Hence $Z(\cI(n,m))=I_{mn}$, and $Z(\cL-\delta\cI(n,m))=Z(\cL)-\delta I_{mn}$.
Thus $\lambda_{mn}(Z(\cL-\delta\cI(n,m)))=\lambda_{mn}(Z(\cL))-\delta \ge 0$ if and ony if  $\lambda_{mn}(Z(\cL))\ge \delta$.
\end{proof}
\begin{corollary}\label{deltaposcp}    Let $\delta\ge 0$ and assume that $\cL:\rH_{n}\to \rH_m$  is $\delta$-positive.  Then $\cL$ is completely positive.  In particular $\cI(n,m)\in CP(n,m)$.
\end{corollary}

Let $\cC:\C^{n\times n}\to \C^{m\times m}$ be a linear transformation over $\C$. 
Then $\cC$ is called rationally represented if the entries  each matrix $\cL(\e_i\e_j^*)$ are Gaussian rationals, denoted as
$\Q+\bi \Q$.  Assume that $\cC$ is rationally represented.  Denote by $\langle \cC\rangle=\sum_{i,j=1}^n \langle\cL(\e_i\e_j^*) \rangle $ the complexity of $\cC$.
\begin{theorem}\label{epsapproxCWcp}  Let $\cA,\cB\in CP(n,m)$ be rationally represented.  Assume furthermore that $\cB$ is $\delta$-positive for
a given rational $\delta>0$.  Then $\hat\rho(\cA,\cB)=\rho(\cA,\cB)$.  
Suppose that $\rho(\cA,\cB)>0$.  Then for any $\varepsilon\in (0,1)\cap \Q$ one can find
 $\tilde\rho(\cA,\cB)\in \Q_{++}$, in polynomial time in $\langle \cA\rangle+\langle \cB\rangle+\langle \delta\rangle+ \langle \varepsilon\rangle $, such that 
\begin{equation}\label{epsapproxCWn2}
\tilde \rho(\cA,\cB)\le \rho(\cA,\cB)\le (1+\varepsilon) \tilde\rho(\cA,\cB).
\end{equation}
\end{theorem}
\begin{proof}  
As $\cB(I_n)\succeq \delta(\tr I_n) I_m=\delta n I_m$, we deduce that  
\noindent
$r(\cA,\cB,I_n)\le (n\delta)^{-1}\rho(\cA(I_n))<t_0\in \Q_{++}$.  Hence $\rho(\cA,\cB)<t_0$.  
As $\cB\succ 0$ part \emph{5} of  Lemma \ref{rhoABachievqc}
yields that $\hat\rho(A,B)=\rho(A,B)$.
As $\hat\rho(\cA,\cB)>0$ Lemma \ref{lowbdcp} yields that $\cA^{\vee}(I_m)\succ 0$.  Hence
$\rho(\cA,\cB)=\hat\rho(\cA,\cB)\ge \rho(\cA^{\vee}(I_m)^{-1}\cB^{\vee}(I_m))^{-1}>s_0\in\Q_{++}$.
Note that the values of $s_0,t_0$  are based on the eigenvalue computations of $\lambda(A,B)$, given by
\eqref{releigAB1}, whose approximation is polynomial \cite{GV13}.

We now start a bisection problem as in the proof of Theorem \ref{epsapproxCWn}.  Suppose that we know that $\rho(\cA,\cB)\in[s_k,t_k]$, where $s_k,t_k\in \Q_{++}$.  Let $t=\frac{s_k+t_k}{2}$.  Consider the minimum problem \eqref{minftprob}.  If $\mu(t)\ge 2\tau$ then $\rho(\cA,\cB)>t$. and we let $s_{k+1}=t,t_{k+1}=t_k$.  Assume now $\mu(t)<2\tau$.
We claim that 
\begin{equation}\label{uprhoABcpest}
\rho(\cA,\cB)\le t+\frac{\sqrt{2\tau}}{\delta}.
\end{equation}
Indeed, let $X(t)$, $Z(t)$ and $W$ be defined as in \eqref{rABcpapr}.  Let $\lambda_1(W)\ge\cdots\ge \lambda_m(W)$ be the $m$-eigenvalues of $W$.  Then
\[2\tau>\|W\|^2=\langle W,W\rangle=\sum_{j=1}^m \lambda_j(W)^2\ge \lambda_m(W)^2.\]
As $\sqrt{2\tau}I_m+W\ge 0$ we obtain
\[(t+\frac{\sqrt{2\tau}}{\delta})\cB(X(t))-\cA(X(t))=Z(t)+\frac{\sqrt{2\tau}}{\delta}\cB(X(t))+W\ge Z(t)+\frac{\sqrt{2\tau}}{\delta}\delta I_m+W\ge 0.\]
Thus $r(\cA,\cB,X(t))\le t+\frac{\sqrt{2\tau}}{\delta}$ which implies \eqref{uprhoABcpest}.  Now choose $\tau=\frac{(s_k+t_k)^2\delta^2}{32}$.
So in the case $\mu(t)<2\tau$ we set $s_{k+1}=s_k, t_{k+1}=\frac{s_k+3t_k}{4}$.   

To conclude we showed that $\rho(\cA,\cB)\in[s_{k+1},t_{k+1}]$, where $[s_{k+1},t_{k+1}]\subset [s_k,t_k]$ and $|t_{k+1}-s_k|\le \frac{3}{4}|t_k-s_k|$.
Hence in polynomial time in $\langle A\rangle+\langle B\rangle+\langle \delta\rangle+ \langle \varepsilon\rangle $ we get the approximation $\tilde\rho(\cA,\cB)\in\Q_{++}$
satisfying \eqref{epsapproxCWn2}.
\end{proof}
\section{Minimal weak optimal solutions for CP-operators}\label{minoptCP}  
For $X\in\rH_{+,n}$ we call $\range X\subseteq\C^n$ the support of $X$, and denote $\supp X=\range X$.  So $\dim\supp X=\rank X$.
Assume that $\cA,\cB\in CP(n,m)$.  Suppose furthermore that $\hat\rho(\cA,\cB)<\infty$.  A weak optimal $Y\in\rH_{+,n}$ is called minimal weak optimal if there is no weak optimal $X\in\rH_{+,n}\setminus\{0\}$ such that $\supp X\subsetneq \supp Y$.

The following result is an analog of Theorem \ref{CWpairthm}.  
\begin{theorem}\label{CWCPpairthm} Let $ m,n$ be positive integers.  Assume that $\cA,\cB\in CP(n,m)$.  Suppose that $\hat\rho(\cA,\cB)\in (0,\infty)$.
 \begin{enumerate} 
 \item Assume that $Y\in\rH_{+,n}\setminus\{0\}$ is weak optimal.  Then at least one of the eigenvalues of $(\cA-\hat\rho(\cA,\cB)\cB)(Y)$ is zero.
 \item Assume that there exists a weak optimal  $Y\in\rH_{+,n}$ whose rank is $\ell\ge 1$.  Let $\V=\range Y$.  Denote by $\cA',\cB':\rS(\V)\to \rH_m$ the restrictions of $\cA,\cB$
 to all $X\in\rH_n$ such that $\range X\subseteq \V$.  Then $\cA',\cB'$ are CP-operators.
 If $\ell\ge m$ then $\rank (\cA'-\hat\rho(\cA,\cB)\cB')<m^2$.
 \item
 A minimal weak optimal $Y$ has rank at most $m$. 
 \item Assume that $Y$ is a minimal weak optimal with rank $m$. Then $Y$ satisfies \eqref{GPFcpop}.  Furthermore $\rank (\cA'-\hat\rho(\cA,\cB)\cB')=m^2-1$.
\item Let $Y'$ be a minimal weak optimal with rank $\ell<m$. Then there exists a minimal weak optimal $Y$, satisfying $\range Y=\range Y'$  with the following property:
The matrix $(\cA-\hat\rho(\cA,\cB)\cB)(Y)$ has at least $\ell$ zero eigenvalues. 
 \end{enumerate}
 \end{theorem}
 \begin{proof} {\it (1)}  Let $\cD(t)=\cA-t\cB$ and $t_0=\hat\rho(\cA,\cB)$.  As $r(\cA,\cB,Y)=t_0$ it follows that $\cD(t_0)(Y)\le 0$.  Suppose to the contrary that $\cD(t_0)(Y)<0$,  As $t_0>0$ there exists $0\le t_1<t_0$ such that $\cD(t_1)(Y)\le 0$.  Hence $r(\cA,\cB,Y)\le t_1<t_0$ contrary to our assumption.

\noindent
 \emph{(2)}  Choose an orthonormal basis $\g_1,\ldots,\g_{\ell}$ of $\V$.  Then $P_{\V}=\sum_{i=1}^{\ell} \g_i\g_i^*$ is the orthogonal projection on $\V$.
 Note that for each $X\in\rH_n$, $P_\V X P_\V\in\rS(\V)$.  (Observe that $P_\V^*=P_\V$.)
 Suppose that $\cC:\rH_n\to\rH_m$ is a CP-operator given by \eqref{defccpop}.  Define $\hat\cC:\rH_n\to\rH_m$ by $\hat\cC(X)=\cC(P_\V XP_\V)$.  Then $\hat  \cC(X)=\sum_{j=1}^k (T_jP_\V)X(T_jP_\V)^*$.  So $\hat \cC$ is also a CP-operator.  Hence $\cA',\cB'$ are CP-operators.  
 By choosing an orthonormal basis $\g_1,\ldots,\g_{\ell}$ of $\V$ we can identify $\rS(\V)$ with $\rH_{\ell}$.  So $\cA',\cB':\rH_{\ell}\to \rH_m$.  Recall that $\dim\rH_{\ell}=\ell^2, \dim\rH_m=m^2$.  Clearly, $\hat\rho(\cA',\cB')=\hat\rho(\cA,\cB)$.  By abusing the notation we assume that an optimal $Y$ is in $\rH_{+,\ell}$.
 
 Assume that $\ell\ge m$.  Let $\cC(t)=\cA'-t\cB':\rH_{\ell}\to \rH_m$.  Thus $\cC(t_0)(Y)=-W$ for some $W\in\rH_{+,m}$.
 Clearly, $\rank\cC(t_0)\le m^2$.   Assume to the contrary that $\rank\cC(t_0)= m^2$.
 Hence the exists an $m^2$ dimensional subspace $\bW\subseteq \rH_{\ell}$ such that $\cC(t_0)|\mathbf{W}$ is an invertible operator.  Let $\cA_1,\cB_1$ and $\cA_2,\cB_2$ be the restrictions of $\cA',\cB'$ to $\mathbf{W}$ and $\mathbf{W} ^{\perp}$ respectively.  Define $\cC_1(t)=\cA_1-t\cB_1, \cC_2(t)=\cA_2-t\cB_2$.  As $\cC_1(t_0)$ is invertible, it follows that there exists $\varepsilon>0$ such that $C_1(t):\bW\to \rH_m$ is invertible for $t\in[t_0-\varepsilon,t_0+\varepsilon]$.  Let $Y=Y_1+Y_2$, where $Y_1\in\mathbf{W}, Y_2\in \mathbf{W}^\perp$.  
 Assume that $Y_1(t)=-\cC_1(t)^{-1}(\cC_2(t)(Y_2)+W)$.  So $Y_1(t_0)=Y_1$.  Hence $\cC(t)(Y_1(t)+Y_2)=-W$.  As $\rank (Y_1+Y_2)=\ell$ it follows that there exits $t_0-\varepsilon <t_1<t_0$ such that $Y_1(t)+Y_2\in\rH_{+,\ell}$ and $\rank (Y_1(t)+Y_2)=\ell$.  Therefore $r(\cA,\cB,Y_1(t)+Y_2)\le t_1<\hat\rho(\cA,\cB)$, which contradicts the definition of $\hat\rho(\cA,\cB)$.
 
 \noindent
 \emph{(3)} Assume to the contrary that $Y$ is a minimal weak optimal with rank $\ell>m$.  Assume as in \emph{2} that we restricted ourselves to $\cA',\cB':\rH_{\ell}\to \rH_m$.
 Let $t_0$ and $\cC(t_0)$ be defined as above.  As $\ell >m$ it follows that $\dim \rH_{\ell}=\ell^2\ge (m+1)^2=m^2 +2m+1$.  So $\dim\ker \cC(t_0)\ge 2m+1\ge 2$.
 Hence there exists an indefinite matrix $X\in\rH_{\ell}$, with at least one positive and one negative eigenvalue such that $\cC(t_0)(X)=0$.  Let $s\in[0,\infty]$ an consider the matrix $Y(s)=Y+sX$.  For $s=0$ $Y\in\rH_{++,\ell}$.  For $s\gg 1$ $Y(s)$ has a negative eigenvalue.  Hence there exists $s_0>0$ such that $Y(s_0)\in\rH_{+,\ell}$ and at least one eigenvalue of $Y(s_0)$ is $0$.  Note that $Y(s_0)\ne 0$ as $Y$ can not be proportional to $X$.  As $\cC(t_0)(Y(s_0))=\cC(t_0)(Y)=-W$ it follows that  $Y(s_0)$ is optimal.
 As $\rank Y(s_0)<\ell$, we deduce that $Y$ is not  minimal weak optimal contrary to our assumptions.

 \noindent 
  \emph{(4)}  Let $Y$ be a minimal weak optimal of rank $m$.  Assume as in \emph{2} that we restricted ourselves to $\cA',\cB':\rH_{m}\to \rH_m$.  
  Part \emph{(2)} yields that $\rank\cC(t_0)<m^2$.  
  Let $X\ne 0$ satisfy $C(t_0)(X)=0$.  If $X$ is indefinite,
  as in the proof of \emph{(2)}, we deduce that $Y$ is not a minimal weak optimal, contrary to our assumption. 
 If $X$ or $-X$ is positive semidefinite then either $X$ or $-X$ is a GPF-vector.  We claim that $Y$ is proportional to $X$.  Otherwise, assuming that $X$ is positive semidefine, by considering $Y-sX$ we deduce that $Y$ is not optimal.  Hence $Y$ is a GPF-vector.  Moreover, the above arguments yield that $\dim\ker \cC(t_0)=1$, i.e.  
 $\rank \cC(t_0)=m^2-1$.
 
 \noindent
 \emph{(5)} Assume that $Y'\in\rH_{+,n}$ is a minimal weak optimal with rank $\ell< m$.  So $(\cA-\rho(A,B)\cB)(Y')\le 0$. 
 By part \emph{1} we know that $(\cA-\rho(A,B)\cB)(Y')$ has at least one zro eigenvalue.  
 If $\ell=1$ the claim \emph{5} of the theorem trivially holds.   
 
 Assume that $\ell>1$.
 Consider all minimal weak optimal $\tilde Y$ such that $\range \tilde Y=\range Y'$.  Let $Y$ be a minimal weak optimal satisfying: $\range Y=\range Y'$ 
 and  $(\cA-\rho(A,B)\cB)(Y)$ has the maximum number of zero eigenvalues.  Assume that this maximum is $p$.  We claim that $p\ge \ell$.  Suppose not.
 As in the proof of part \emph{(2)} we restrict ourselves to $\cA',\cB':\rH_{\ell}\to \rH_m$.  So $t_0=\rho(\cA',\cB')$.
 By abusing the notation we assume that $Y',\tilde Y, Y\in\rH_{+,\ell}$.  Let $W=-C(t_0)(Y)$.  Then $\C^m= \U_1\oplus \U_2$, where $\U_1=\range W, \U_2=\U_1^\perp=\ker W$.  Note that $W|\U_1$ positive  definite.  Let $P_{\U_2}$ be the orthogonal projection of $\C^m$ on $\U_2$. Let $\tilde\cA,\tilde\cB:\rH_{\ell}\to \rH_p$, where we identify
 $\rH_p$ with $P_{\U_2}\rH_mP_{\U_2}$ and $\tilde \cA,\tilde\cB$ with $P_{\U_2}\cA P_{\U_2},P_{\U_2}\cB P_{\U_2}$ respectively.  Clearly, $\hat\rho(\tilde\cA,\tilde\cB)\le \hat\rho(\cA',\cB')$.  As in the proof of part \emph{(5)} of Theorem \ref{CWpairthm} we claim that $\hat\rho(\tilde\cA,\tilde\cB)= \hat\rho(\cA',\cB')$.  Suppose not.  Then there exists $U\in \rH_{+,\ell}\setminus\{0\}$ such that $r(\tilde \cA,\tilde\cB,U)=t_1<t_0$. Clearly, $Y$ and $U$ are linearly independent.  As $Y\in\rH_{++,\ell}$ we can assume that $Y\ge U$.
 Let $Y(f)=(1-f)Y+fU$.  As in the proof of part \emph{(5)} of Theorem \ref{CWpairthm} we claim that there exists $\varepsilon>0$ such that $r(\tilde\cA,\tilde\cB, (1-f)Y+fU))\le t_0-\varepsilon f$ for $f\in[0,1]$.  We can choose $\varepsilon=(t_0-t_1)\lambda'$, where $\lambda'$ is the smallest positive eigenvalue of $\tilde \cB(Y)^{-1}\tilde\cB(U)$.  (Note that the eigenvalues of 
 $\tilde \cB(Y)^{-1}\tilde\cB(U)$ are the eigenvalues of the Hermitian matrix $\tilde \cB(Y)^{-\frac{1}{2}}\tilde\cB(U)\tilde \cB(Y)^{-\frac{1}{2}}$.) 
 We now claim that there exists small positive $\delta$ such that for 
  we have the inequality 
  \begin{equation}\label{lambAB'ineq}
  r(\cA',\cB',(1-f)Y+fU))\le t_0-\frac{1}{2}\varepsilon f, \quad f\in[0,\delta].
  \end{equation}
    For that it is enough to show that $\lambda(\cA'((1-f)Y+fU),\cB'((1-f)Y+fU))\le t_0-\frac{1}{2}\varepsilon f$ for $f\in[0,\delta]$. 
 We use the first variation formula for  a geometrically simple eigenvalue $t_0$ of the Rayleigh ratio on the right-hand side of \eqref{releigAB1}, where $A=A(f)=\cA'((1-f)Y+fU))$ and $B=B(f)=\cB'((1-f)Y+fU))$, as a function in $f$.  The standard variation formula for a geometrically simple eigenvalue \cite[\S3.8]{Fri15} yields that it is enough to consider the first variation of the Rayleigh quotient given by the right-hand side of \eqref{releigAB1} for $\x\in\U_2\setminus\{\0\}$.   This is equivalent to consider $\lambda(\tilde\cA((1-f)Y+fU),\tilde\cB((1-f)Y+fU))$.  The inequality $\lambda(\tilde\cA((1-f)Y+fU),\tilde\cB((1-f)Y+fU))\le t_0-\varepsilon f$ for $f\in[0,\delta]$ yields that 
 \begin{equation}\label{firstderineq}
 \frac{d}{df}\lambda(\cA'((1-f)Y+fU),\cB'((1-f)Y+fU))|_{f=0^+}\le -\varepsilon.
 \end{equation}
 Recall Rellich's theorem that the eigenvalues of analytic functions of Hermitian matrices are analytic in the neighborhood of $\R$ \cite[4.17]{Fri15}, (when we do not insist on ordering of the eigenvalues).  Hence $\lambda(\cA'((1-f)Y+fU),\cB'((1-f)Y+fU))$ is analytic in $[0,\delta]$ for some $\delta>0$.  That is,  it has convergent Taylor series at $f=0$.  In particular it is in the class $C^2[0,\delta]$.  Hence \eqref{firstderineq} yields \eqref{lambAB'ineq} for small enough positive $\delta$.  
 
 Clearly, \eqref{lambAB'ineq}
 contradicts our assumption that $\hat\rho(\cA',\cB')=t_0$.  Therefore $\hat\rho(\tilde\cA,\tilde\cB)=t_0=r(\tilde\cA,\tilde\cB,Y)$.  Let $\tilde\cC(t)=\tilde\cA-t\tilde\cB$.  
 So $\tilde\cC(t_0):\rH_{\ell}\to \rH_p$.  Our assumption that $\ell>p$ yields that $\ker \tilde\cC(t_0)\ge \ell^2-p^2\ge 3$.  In particular, there exists an indefinite $\tilde X\in\rH_{\ell}$ 
 such that $\tilde \cC(t_0)(\tilde X)=0$.  Identify $\tilde X$ with $X\in\rH_n$, were $\range X\subseteq \range Y$.  The assumption that $\tilde\cC(t_0)(\tilde X)=0$ is equivalent to $\range\cC(t_0)(X)\subseteq \range \cC(t_0)(Y)=\range W$.  Let $Y(s)=Y+sX$.  Then $\cC(t_0)(Y(s))=-W+s\cC(t_0)(X)$.  As $\range\cC(t_0)(X)\subseteq \range \cC(t_0)(Y)=\range W$ it follows that $-W+s\cC(t_0)(X)\le 0$  for $s\ge 0$ if and only if $s\in [0,s_1]$, where $s_1\in (0,\infty]$.  Let $s_0>$ be the largest $s>0$ such that $Y(s)\ge 0$.  So $\rank Y(s_0)\in [\ell-1]$. Note that $Y(s)$ is optimal for $s\in [0,\min(s_0,s_1)]$. We claim that $s_1<s_0$.  Suppose not.  Then $Y(s_0)$ is weak optimal, $\range Y(s_0)\subset \range Y$ and rank $Y(s_0)\le \ell-1$.  Hence $Y$ is not minimal weak optimal, contrary to our assumptions.  Consider a minimal weak optimal $Y(s_1)$.  Note that
 $\rank (-W+s_1\cC(t_0)(X))<\rank -W=m-p$.   Hence $\ker \cC(t_0)(Y(s_1)\ge p+1$, contrary to the choice of $Y$.  Therefore $p\ge \ell$.  
 \end{proof}

\renewcommand{\abstractname}{Acknowledgements}
\begin{abstract} The author thanks the referee for his comments.
 Shmuel Friedland was partially supported by Simons collaboration grant for mathematicians.
\end{abstract}


\bibliographystyle{abbrv}

\end{document}